\documentclass[preprint,12pt]{elsarticle}
\usepackage{amssymb}
 \usepackage{amsthm}
 \usepackage{amsmath}
\usepackage{mathrsfs}
\usepackage{bbm}
\usepackage[top=1in, bottom=1in, left=1in, right=1in]{geometry}

\newtheorem{thm}{Theorem}
\newtheorem{lem}[thm]{Lemma}
\newtheorem{corollary}[thm]{Corollary}
\newtheorem{proposition}[thm]{Proposition}
\newdefinition{rmk}{Remark}
\newdefinition{defi}{Definition}
\newdefinition{example}{Example}

\newcommand{\identity}{I}
\newcommand{\unitvector}{\mathbbm{e}}

\journal{Linear Algebra and Applications}

\begin{document}

\begin{frontmatter}

%% Title, authors and addresses

%% use the tnoteref command within \title for footnotes;
%% use the tnotetext command for theassociated footnote;
%% use the fnref command within \author or \address for footnotes;
%% use the fntext command for theassociated footnote;
%% use the corref command within \author for corresponding author footnotes;
%% use the cortext command for theassociated footnote;
%% use the ead command for the email address,
%% and the form \ead[url] for the home page:
%% \title{Title\tnoteref{label1}}
%% \tnotetext[label1]{}
%% \author{Name\corref{cor1}\fnref{label2}}
%% \ead{email address}
%% \ead[url]{home page}
%% \fntext[label2]{}
%% \cortext[cor1]{}
%% \affiliation{organization={},
%% addressline={},
%% city={},
%% postcode={},
%% state={},
%% country={}}
%% \fntext[label3]{}

\title{The $d$-Majorization Polytope}

% use optional labels to link authors explicitly to addresses:
 \author[label1,label2]{Frederik vom Ende}
 \affiliation[label1]{organization={Department of Chemistry},
 addressline={Technische Universit{\"a}t M{\"u}nchen},
 city={Garching},
 postcode={85737},
% state={},
 country={Germany}}

 \affiliation[label2]{organization={Munich Centre for Quantum Science and Technology (MCQST)},
% addressline={},
 city={M{\"u}nchen},
 postcode={80799},
% state={},
 country={Germany}}
 
 \author[label3]{Gunther Dirr}
 \affiliation[label3]{organization={Department of Mathematics},
 addressline={University of W{\"u}rzburg},
 city={W{\"u}rzburg},
 postcode={97074},
% state={},
 country={Germany}}
\begin{abstract}
  We investigate geometric and topological properties of $d$-majorization -- a generalization of classical majorization
  to positive weight vectors $d \in \mathbb{R}^n$. In particular, we derive a new, simplified characterization of
  $d$-majorization which allows us to  work out a halfspace description of the corresponding $d$-majorization polytopes.
  That is, we write the set of all vectors which are $d$-majorized by some given vector $y \in \mathbb{R}^n$ as an  intersection of
  finitely many half spaces, i.e.~as solutions to an inequality of the type $Mx\leq b$. Here $b$ depends on $y$ while $M$ can be
  chosen independently of $y$. This description lets us prove continuity of the $d$-majorization polytope (jointly
  with respect to $d$ and $y$) and, furthermore, lets us fully characterize its extreme points. Interestingly,
  for $y\geq 0$ one of these extreme points classically majorizes every other element of the $d$-majorization polytope.

    Moreover, we show that the induced preorder structure on $\mathbb{R}^n$ admits minimal and maximal elements. While the
    former are always unique the latter are unique if and only if they correspond to the unique minimal
    entry of the $d$-vector.
  % We give a new characterization of $d$-majorization \textcolor{red}{-- a generalization of classical majorization to positive weight vectors $d$ --
  % on real vectors and investigate its geometric and topological properties as well as its induced preorder structure}. In particular, we show that the
  % induced preorder admits minimal and maximal elements. While the former are always unique the latter are unique if and only if they correspond to
  % the unique minimal entry of the $d$-vector.
  % %
  % Moreover we work out a halfspace description of $d$-majorization, that is, we write
  % the set of all vectors which are $d$-majorized by some given vector $y$ as an
  % intersection of finitely many half spaces, i.e.~by an inequality of the type $Mx\leq b$.
  % Here $b$ depends on $y$ while $M$ can be chosen independently of $y$. This description
  % lets us prove continuity of the resulting $d$-majorization polytope (jointly with respect
  % to $d$ and $y$) and, furthermore, lets us fully characterize its extreme points. Interestingly,
  % if $y\geq 0$ then one of the extreme points of this polytope classically majorizes every element
  % from the polytope. 
\end{abstract}

%%%Graphical abstract
%\begin{graphicalabstract}
%%\includegraphics{grabs}
%\end{graphicalabstract}

%%Research highlights
%\begin{highlights}
%\item Research highlight 1
%\item Research highlight 2
%\end{highlights}

\begin{keyword}
majorization relative to $d$ \sep $d$-majorization polytope \sep convex polytopes \sep extreme points
%% keywords here, in the form: keyword \sep keyword

%% PACS codes here, in the form: \PACS code \sep code

%% MSC codes here, in the form: \MSC code \sep code
%% or \MSC[2008] code \sep code (2000 is the default)
\MSC[2020] 15B51\sep 26A51\sep 52B11\sep 52B12
\end{keyword}

\end{frontmatter}

%% \linenumbers

%% main text
\section{Introduction}
\noindent
The concept of $d$-majorization is a natural generalization of the classical notion of majorization as we
  will see in a moment. But let us first describe a deep relation between $d$-majorization,
$d$-stochastic matrices, and quantum physics. Readers who are not
familiar with the relevant quantum-mechanical terminology 
 can skip the following two paragraphs as---although we will draw some connections to 
 the physics literature throughout this article---we will not use any of these results.
 
Over the last few years, sparked by Brand\~ao, Horodecki, Oppenheim 
\cite{Brandao15,Horodecki13}, and further pursued by others
\cite{Faist17,Gour15,Lostaglio18,Sagawa19,Mazurek19,Alhambra19}, thermo-majorization and in particular its
 resource theory approach % to quantum thermodynamics
 has been a widely discussed and researched topic in quantum physics. Here the 
 central question is:
 Given a fixed ``background temperature'' $T>0$ as well as initial and target states 
 (density operators) of a quantum system,
 % with Hamiltonian (hermitian operator) $H_S$,
 can the former be mapped to the 
 latter by means of a thermal operation?
Here the set of all thermal operations which constitutes a compact, convex semigroup
   within the set of all quantum channels \cite[App.~C] {Lostaglio15} consists
 of all linear maps which can be approximated arbitrarily well by quantum channels 
 of the form
$$
\Phi(\cdot)=\operatorname{tr}_B\Big(U\Big((\cdot)\otimes\frac{e^{-H_B/T}}{\operatorname{tr}(e^{-H_B/T})}\Big)U^*\Big)
\,,
$$
where $H_B$ is an arbitrary ``bath'' Hamiltonian, $H_S$ the system's Hamiltonian, $U$ any unitary
  operator satisfying the stabilizer condition
  $
  U (H_S\otimes\mathbbm{1}_B+\mathbbm{1}_S\otimes H_B) U^* = H_S\otimes\mathbbm{1}_B+\mathbbm{1}_S\otimes H_B
  $, 
and $\operatorname{tr}_B$ denotes the partial trace operation with respect to the
``bath'', cf.~\cite{Janzing00,Brandao13}.

Two key properties of thermal operations are that $e^{-H_S/T}$ is a fixed point of all 
thermal operations---that is, the Gibbs 
state of the system is preserved---and that they commute with the system's ``natural'' 
dynamics at all times, i.e.
\begin{equation}\label{eq:intro_1}
\Phi(e^{-iH_St}(\cdot)e^{iH_St})=e^{-iH_St}\Phi(\cdot)e^{iH_St}\,,
\end{equation} 
for all $t \geq 0$ which is equivalent to $[\Phi,\operatorname{ad}_{H_S}]=0$ \cite{Lostaglio15_2}. 
However, these two features do \textit{not} fully characterize thermal operations, 
meaning there exist quantum channels which satisfy the above properties but lie outside the set of
thermal operations \cite{Ding21}. For a comprehensive introduction to this topic we refer to the review article \cite{Lostaglio19}.

For finite-dimensional systems, the above commutation relation \eqref{eq:intro_1} allows for a partial answer
to the state-conversion problem posed above because it reveals that $\Phi$ and $\operatorname{ad}_{H_S}$
  share common invariant subspaces. In particular, if $H_S \in \mathbb{C}^{n \times n}$ is diagonal with non-degenerate
  spectrum the set of diagonal density matrices is a common invariant and the action of thermal operations on it
% the diagonal entries
% (w.r.t.~an eigenbasis of $H_S$)
% of a density matrix
  coincides with the action of $d$-stochastic\footnote{A matrix $A\in\mathbb R^{n\times n}$ is said to be $d$-\textit{stochastic}
    % i.e.~column-stochastic matrices $A$ with $Ad=d$ \textit{column-stochastic}
    if it is column-stochastic with $Ad=d$, where \textit{column-stochastic} means that all its entries are non-negative and
    % $\sum\nolimits_{i=1}^n A_{ij}=1$ for all $j=1,\ldots,n$, so the entries of
    each column sums up to one.}
  matrices on $\mathbb{R}^n$, cf.~\cite[Thm.~1]{Lostaglio19}. The vector $d$ then consists of the diagonal\footnote{Certainly,
    it suffices to require that the density matrices are represented in an eigenbasis of $H_S$ in which case $d$ consists of
    the eigenvalues of $e^{-H_S/T}$.} entries of the matrix $e^{-H_S/T}$ (up to a global constant which can be disregarded). Thus,
  if the initial state and the final state are both ``diagonal'' then the $d$-stochastic matrices fully characterize all possible
  state transitions \cite{Horodecki13,Korzekwa16}.

From a mathematical point of view this puts us in the realm of majorization relative to a 
positive vector $d\in\mathbb R^n$ as introduced by Veinott \cite{Veinott71} 
and, in the quantum regime, by Ruch, Schranner, and Seligman \cite{Ruch78}.
For positive $d$, some vector 
$x$ is said to be $d$-majorized by $y$, denoted by $x\prec_d y$, if there exists a
$d$-stochastic matrix $A$ such that $x=Ay$. A variety of characterizations of $\prec_d$ 
and $d$-stochastic matrices can be found in the work of Joe \cite{Joe90} and in 
Prop.~\ref{lemma_char_d_vec} below.

Certainly, the concept of classical majorization as first introduced by Muirhead \cite{Muirhead02} and more widely
 spread by Hardy, Littlewood, and P{\'o}lya \cite{Hardy52} is a special case of $d$-majorization.
% roughly speaking describes if a vector with real entries is ``less or more nearly equal'' than another, and found numerous
% applications in various fields of science \cite{Lorenz05,Dalton20,Parker80,BSH16_alt,OSID17,OSID19}.}
 More precisely, one says that a vector $x\in\mathbb R^n$ is classically majorized by $y\in\mathbb R^n$, denoted
 by $x\prec y$, if $\sum\nolimits_{i=1}^n x_i=\sum\nolimits_{i=1}^n y_i$ and $\sum\nolimits_{i=1}^j x_{[i]}\leq \sum\nolimits_{i=1}^j y_{[i]}$
for all $j=1,\ldots,n-1$, where $x_{[i]},y_{[i]}$ are the components of $x,y$ in decreasing order.
This is well known to be equivalent to the existence of a doubly-stochastic matrix $A$, that is, a $d$-stochastic matrix
with $d=(1,\ldots,1)^\top $, such that $x=Ay$ \cite[Thm.~46]{Hardy52}. A comprehensive survey on classical majorization can be
found in \cite{MarshallOlkin}. For numerous applications in various fields of science we also refer to
\cite{Lorenz05,Dalton20,Parker80,BSH16_alt,OSID17,OSID19}.

%\begin{center}
%$\mathscr H$-Description of Classical Majorization
%\end{center}
%
%
%
%In order to generalize majorization to arbitrary weight vectors it, unsurprisingly, is advisable to first recap and explore classical vector majorization. 
%The common definition of vector majorization goes as follows: Given $x,y\in\mathbb R^n$ one says $x$ is majorized by 
%$y$, denoted by $x\prec y$, if $\sum\nolimits_{i=1}^n x_i=\sum\nolimits_{i=1}^n y_i$ and 
%\begin{equation}\label{eq:maj_cond}
%\sum\nolimits_{i=1}^j x_{[i]}\leq \sum\nolimits_{i=1}^j y_{[i]}\quad\text{ for all }j=1,\ldots,n-1
%\end{equation}
%where $x_{[i]},y_{[i]}$ 
%are the components of $x,y$ in decreasing order, respectively. Given how well-explored this concept is there are numerous
%characterizations for $\prec$, cf.~\cite[Ch.~1, Point A.3]{MarshallOlkin}. The most notable one for our purposes is 
%the following: $y$ majorizes $x$ if and only if there exists a doubly stochastic matrix, i.e.~a matrix $A\in\mathbb R_+^{n\times n}$ which satisfies $\unitvector^{\top}A=\unitvector^{\top}$ and $A\unitvector=\unitvector$, such that $x=Ay$. 
Interestingly, for $x,y\in\mathbb R^n$ one has $x\prec y$ if and only if $x$ lies in the convex hull of all permutations of $y$ (as shown in, e.g., \cite{Rado52}, or as a direct consequence of Birkhoff's theorem 
\cite[Ch.~2, Thm.~A.2]{MarshallOlkin}). Therefore the set $\{x\in\mathbb R^n : x\prec y\}$ is a convex polytope with at most $n!$ corners.
Hence it has a half-space description, that is, it can be written as the intersection of finitely many half-spaces
  or equivalently as the solution to finitely many linear (in-)equalities. The precise result as first stated in \cite[Thm.~1]{Dahl10} reads as follows: Given $y\in\mathbb R^n$ one has
$$
\{x\in\mathbb R^n : x\prec y\}=\Big\{x\in\mathbb R^n : \Big(\sum_{j=1}^n x_i=\sum_{j=1}^n y_i\Big)\wedge \Big(\forall_{m\in\{0,1\}^n}\ m^\top x\leq \sum_{i=1}^{m_1+\ldots+m_n}y_{[i]}\Big)\Big\}\,.
$$
This result motivated us to work out a half-space description of ``the'' $d$-majorization polytope to study its extreme points as well as further topological properties.

In the main, this manuscript is concerned with analyzing the preorder structure of $d$-majorization and its geometry.
  % from the viewpoint of convex polytopes.
  It is organized as follows:
First, in Section \ref{sec_dmaj} we extend the list of existing characterizations of $d$-majorization
by a novel one, which allows to check $d$-majorization in finitely many steps (Prop.~\ref{lemma_char_d_vec} 
(vi)), and we identify the minimal and maximal elements of this preorder (Thm.~\ref{prop_1}).
In Section \ref{sec_prelim_vector} we briefly revisit convex polytopes and their 
equivalent descriptions 
%before applying this to classical majorization on real vectors 
%in Section \ref{sec_maj_H}. This yields a halfspace description ($\mathscr H$-description) of majorization 
%(Prop.~\ref{prop_maj_halfspace}) and enables a new proof of the well-known 
%result that a vector $y$ majorizes a vector $x$ if and only if $x$ lies within the convex 
%hull of all permutations of $y$ (Coro.~\ref{coro_classical_maj_extreme}).
%
%Then in 
before studying $d$-majorization from this perspective in Section \ref{sec_d_maj_poly}.
  In particular, the novel characterization of Prop.~\ref{lemma_char_d_vec} allows us to work out a half-space description
of ``the'' corresponding $d$-majorization polytope
% of the set of all vectors $d$-majorized by some fixed vector $y$,
(Thm.~\ref{thm_maj_halfspace}) and to prove its continuity with respect to $d$ and $y$ (Thm.~\ref{thm_cont_Md}).
Moreover, we identify its extreme points (Thm.~\ref{thm_Eb_sigma}) and -- as for classical majorization -- we show that
the number of extreme points is always upper bounded by $n!$ (Coro.~\ref{thm_convex_poly}). Finally we conclude
that if the ``initial'' vector $y$ is non-negative, then one of the extreme points majorizes every element of 
the polytope \textit{classically} (Thm.~\ref{theorem_max_corner_maj}).

\section{Characterizations and Preorder Properties of $d$-Majorization}\label{sec_dmaj}
For the purpose of this paper be aware of the following notions and notations:
\begin{itemize}
\item In accordance with Marshall and Olkin \cite{MarshallOlkin}, $\mathbb R_+^n$ ($\mathbb R_{++}^n$) denotes the set of all real vectors with non-negative (strictly positive) entries. Whenever it is clear that $x$ is a real vector of length $n$ we occasionally write $x\geq 0$ ($x>0$) to express non-negativity (strict positivity) of its entries.
\item $\unitvector$ shall denote the column vector of ones, i.e.~$\unitvector=(1,\ldots,1)^\top $.
\item $\|\cdot\|_1$ is the usual $1$-norm on $\mathbb R^n$ (or $\mathbb C^n$).
\item $S_n$ is the symmetric group (the group of all permutations of $\{1,\ldots,n\}$).
\item The standard simplex $\Delta^{n-1}\subseteq\mathbb R^n$ is given by the convex hull of all standard basis vectors $e_1,\ldots,e_n$ and precisely contains all probability vectors, that is, all vectors $x\in\mathbb R_+^n$ with $\unitvector^{\top}x=1$.
\end{itemize}
Having reviewed classical vector majorization as well as its polytope properties in the introduction, let us now dive into the ``non-symmetric'' case of majorization, that is, the case where doubly stochastic matrices (having $\unitvector$ as ``left''
and ``right'' fixed point) are replaced by $d$-stochastic matrices (having $\unitvector$ and $d \in \mathbb R_{++}^n$ as ``left''
and ``right'' fixed point, respectively).
As also explained in the introduction this concept is closely related to thermo-majorization and quantum thermodynamics in general, and we will occasionally point out some connections to the physics literature if appropriate. For the following definition we mostly\footnote{
Usually $d$-stochastic matrices are defined via $d^\top A=d^\top $ and $A\unitvector=\unitvector$ which is equivalent to the definition below as it only differs by transposing once. This is because we consider $d,x,y$ to be usual column vectors whereas \cite{Joe90,MarshallOlkin} consider row vectors.\label{footnote_transpose_def}
}
follow \cite[p.~585]{MarshallOlkin}. 
\begin{defi}\label{defi_d_stochastic_matrix}
Let $d\in\mathbb R_{++}^n$ and $x,y\in\mathbb R^n$ be given. A square matrix $A\in\mathbb R^{n\times n}$ is said to be \textit{column-stochastic} if $A_{ij}\geq 0$ for all $i,j=1,\ldots,n$ (i.e.~$A\in\mathbb R_+^{n\times n}$) and $\unitvector^{\top}A=\unitvector^{\top}$. If, additionally, $Ad=d$ then $A$ is said to be \textit{$d$-stochastic}. The set of all $d$-stochastic $n\times n$ matrices is denoted by $ s_d(n)$.
Moreover, $x$ is said to be \textit{$d$-majorized} by $y$, denoted by $x\prec_d y$, if there exists $A\in s_d(n)$ such that $x=Ay$.
\end{defi}
In particular, $x\prec_d y$ implies
$
\unitvector^{\top}x=\unitvector^{\top}Ay=\unitvector^{\top}y
$. Also note that this definition of $\prec_d$ naturally generalizes to complex vectors, cf.~also \cite{Goldberg77}.
\begin{rmk}\label{rem_maj_vector}
\begin{itemize}
\item[(i)] For any $d\in\mathbb R_{++}^n$, the set $ s_d(n)$ constitutes a convex, compact subsemigroup of $\mathbb C^{n\times n}$ with identity element $\identity_n$. In particular it acts contractively in the $1$-norm:
for all $z\in\mathbb R^n$ and $A\in\mathbb R_+^{n\times n}$ with $\unitvector^{\top}A=\unitvector^{\top}$ one has the estimate
\begin{equation}\label{eq:doubly_stoch_trace_norm}
\|Az\|_1=\sum_{i=1}^n\Big|\sum_{j=1}^n A_{ij}z_j\Big|\leq \sum_{i,j=1}^n A_{ij}|z_j|=\sum_{j=1}^n \Big(\sum_{i=1}^n A_{ij}\Big) |z_j|=\sum_{j=1}^n|z_j|=\|z\|_1\,.
\end{equation}
%In particular this holds for all doubly stochastic matrices.
\item[(ii)] 
 Note that $d$-majorization is a special case of so-called matrix majorization: Given matrices $A\in\mathbb R^{m\times p}$, $B\in\mathbb R^{n\times p}$ -- keeping in mind footnote \ref{footnote_transpose_def} -- one says $A$ majorizes $B$ (denoted by $A\prec B$) if there exists $X\in\mathbb R_+^{m\times n}$ with $\unitvector^\top X=\unitvector^\top $ such that $XB=A$ \cite{Dahl99a,Dahl99b}. 
With this one recovers $d$-majorization by setting $B=(d\ y)$, $A=(d\ x)$ because then $A\prec B$ holds iff $x\prec_dy$ as is readily verified.
\item[(iii)] By Minkowski's theorem \cite[Thm.~5.10]{Brondsted83}, the previous point implies that $s_d(n)$ can be written as the convex hull of its extreme points. However---unless $d=\unitvector$---this does not prove to be all too helpful as stating said extreme points (for $n>2$) becomes quite delicate\footnote{The number of extreme points of $s_d(n)$ is lower bounded by $n!$ and upper bounded by $\binom{n^2}{2n-1}$, cf.~\cite[Rem.~4.5]{Joe90}.}. To substantiate this the extreme points for $n=3$ and non-degenerate $d\in\mathbb R_{++}^3$ can be found in Lemma \ref{lemma_extreme_points} (\ref{app_a}).
%On the other hand, the previous results suggest that $M_d(y)$ behaves more nicely in this regard.
\item[(iv)] If some entries of the $d$-vector coincide, then $\prec_d$ is known to be a preordering but not a partial ordering. Contrary to what is written in \cite[Rem.~4.2]{Joe90} this in general does \textit{not} change if the entries of $d$ are pairwise distinct: To see this, consider $d=(3,2,1)^\top $, $x=(1,0,0)^\top $, $y=(0,\tfrac23,\tfrac13)^\top $, and
$$
A=\begin{pmatrix} 0&1&1\\ \frac23&0&0\\\frac13&0&0 \end{pmatrix}\in s_d(3)\,.
$$
Then $Ax=y$ and $Ay=x$ so $x\prec_d y\prec_d x$, but obviously $x\neq y$.
%This counterexample can actually be easily modified to any $d\in\mathbb R^3_{++}$ with $d_1=d_2+d_3$.
\end{itemize}
\end{rmk}

%The following result is simple yet important for later on, a proof can be found in \ref{app_a0}.
%\begin{lem}\label{lemma_stoch_transf}
%Let $x,y\in\mathbb R^n$ such that $\unitvector^{\top}x=\unitvector^{\top}y$ and $\sum_{j=1}^n |x_j|\leq\sum_{j=1}^n |y_j|$% (so $\|a\|_1\leq\|b\|_1$)
%. Then there exists a column-stochastic matrix $A\in\mathbb R^{n\times n}$ such that $Ay=x$. %If $b\neq 0$ then $X$ can even be chosen such that
%%$
%%\rank X\geq 1+\big|\{j : b'_j=0\}\big|\,.
%%$
%\end{lem}

Now let us summarize the known characterizations of $\prec_d$: While the equivalences of (i) through (v) in the following proposition are due to Joe \cite[Thm.~2.2]{Joe90}, number (vi) will be a new result of ours. Moreover, (vii) is related to the definition most prominent among the physics literature, called ``thermo-majorization curves'' \cite{Horodecki13}. Indeed the criterion (vii) we present here is a more explicit version of \cite[Thm.~4]{Alhambra16}; for more on this, cf.~Ch.~\ref{sec_d_poly_analysis}.
\begin{proposition}\label{lemma_char_d_vec}
Let $d\in\mathbb R_{++}^n$ and $x,y\in\mathbb R^n$ be given. The following are equivalent.
\begin{itemize}
\item[(i)] $x\prec_dy$
\item[(ii)] $\sum_{j=1}^n d_j \psi(\frac{x_j}{d_j})\leq \sum_{j=1}^n d_j \psi(\frac{y_j}{d_j})$ for all continuous, convex functions $\psi:D(\psi)\subseteq\mathbb R\to\mathbb R$ such that $\{\frac{x_j}{d_j} : j=1,\ldots,n\},\{\frac{y_j}{d_j} : j=1,\ldots,n\}\subseteq D(\psi)$.
\item[(iii)] $\sum_{j=1}^n (x_j-td_j)_+\leq\sum_{j=1}^n (y_j-td_j)_+$ for all $t\in\mathbb R$ where $(\cdot)_+:=\max\{\cdot,0\}$.
\item[(iv)] $\sum_{j=1}^n (x_j-td_j)_+\leq\sum_{j=1}^n (y_j-td_j)_+$ for all $t\in\{ \frac{x_i}{d_i},\frac{y_i}{d_i} : i=1,\ldots,n \}$.
\item[(v)] $\|x-td\|_1\leq\|y-td\|_1$ (i.e.~$\sum_{j=1}^n |x_j-td_j|\leq\sum_{j=1}^n |y_j-td_j|$) for all $t\in\mathbb R$.
\item[(vi)] $\unitvector^{\top}x=\unitvector^{\top}y$ and $\|x-\frac{y_i}{d_i}d\|_1\leq\|y-\frac{y_i}{d_i}d\|_1$ for all $i=1,\ldots,n $.
\item[(vii)] $\unitvector^{\top}x=\unitvector^{\top}y$ and for all $j=1,\ldots,n-1 $
$$
\sum\nolimits_{i=1}^jx_{\sigma(i)}\leq \min_{i=1,\ldots,n}\Big( \unitvector^{\top}\Big(y-\frac{y_i}{d_i}d\Big)_++\frac{y_i}{d_i}\Big( \sum\nolimits_{k=1}^jd_{\sigma(k)}\Big)\Big)
$$
where $\sigma\in S_n$ is any permutation such that $\frac{x_{\sigma(1)}}{d_{\sigma(1)}}\geq\ldots\geq \frac{x_{\sigma(n)}}{d_{\sigma(n)}}$. 
\end{itemize}
\end{proposition}
\begin{proof}
(v) $\Rightarrow$ (vi): For $t$ large enough all entries of $x-td,y-td$ are non-positive so
$$
-\unitvector^{\top}(x-td)=\|x-td\|_1\leq \|y-td\|_1=-\unitvector^{\top}(y-td)
$$
and thus $\unitvector^{\top}x\geq \unitvector^{\top}y$. Doing the same for $-t$ large enough gives $\unitvector^{\top}x\leq \unitvector^{\top}y$ so together $\unitvector^{\top}x= \unitvector^{\top}y$.

(vi) $\Rightarrow$ (v): Define $P:=\{ \frac{x_i}{d_i},\frac{y_i}{d_i} : i=1,\ldots,n \}$; w.l.o.g.~$|P|>1$. As argued before, $\unitvector^{\top}x=\unitvector^{\top}y$ implies $\|x-td\|_1=\|y-td\|_1$ on $t\in(-\infty,\min P]\cup[\max P,\infty)$. Now define
\begin{align*}
g_x:[\min P,\max P]&\to\mathbb R_+\\
t&\mapsto \|x-td\|_1=\sum\nolimits_{i=1}^n d_i \Big|\frac{x_i}{d_i}-t\Big|
\end{align*}
and $g_y$ analogously. Thus all that is left to show is $g_x(t)\leq g_y(t)$ for all $t\in[\min P,\max P]$.

First note that $g_x(
\min P)=g_y(\min P)$, $g_x(\max P)=g_y(\max P)$, and $g_x(\frac{y_i}{d_i})
\leq g_y(\frac{y_i}{d_i})$ for all $i=1,\ldots,n$ by assumption, meaning we can choose $t\in 
(\min P,\max P)\setminus\{\frac{y_1}{d_1},\ldots,\frac{y_n}{d_n}\}$. Hence there exists $i=0,\ldots,n$ such that $
\frac{y_i}{d_i}<t<\frac{y_{i+1}}{d_{i+1}}$ where $\frac{y_0}{d_0}:=\min P$, $\frac{y_{n+1}}{d_{n+1}}:=\max P$. Defining $\lambda:=(\frac{y_{i+1}}{d_{i+1}}-t)/(\frac{y_{i+1}}{d_{i+1}}-\frac{y_i}{d_i})\in(0,1)$ and using convexity of $g_x$ we compute
\begin{align*}
g_x(t)=g_x\Big(\lambda\frac{y_i}{d_i}+(1-\lambda)\frac{y_{i+1}}{d_{i+1}}\Big)&\leq \lambda g_x\Big(\frac{y_i}{d_i}\Big)+(1-\lambda)g_x\Big(\frac{y_{i+1}}{d_{i+1}}\Big)\\
&\leq\lambda g_y\Big(\frac{y_i}{d_i}\Big)+(1-\lambda)g_y\Big(\frac{y_{i+1}}{d_{i+1}}\Big)\\
&=g_y\Big(\lambda\frac{y_i}{d_i}+(1-\lambda)\frac{y_{i+1}}{d_{i+1}}\Big)=g_y(t)\,.
\end{align*}
In the last line we used that $g_y$ is affine linear on each interval $[\frac{y_i}{d_i},\frac{y_{i+1}}{d_{i+1}}]$.
% Evidently, $g_x$ is convex, we have , and $g_y$ is continuous piecewise linear with change in slope only if $t=\frac{y_i}{d_i}$ for some $i=1,\ldots,n$. But at those changes in slope we by assumption have $g_x(\frac{y_i}{d_i})\leq g_y(\frac{y_i}{d_i})$ so Lemma \ref{lemma_convex_cpl_compare} (ii) implies $g_x(t)\leq g_y(t)$ for all $t\in[\min P,\max P]$, and thus $\|x-td\|_1\leq\|y-td\|_1$ for all $t\in\mathbb R$.

%%ARXIV ONLY
As stated before, the equivalence of (i) through (v) is due to \cite[Thm.~2.2]{Joe90}. However for the sake of this work being self-contained (and possibly filling some gaps in the literature) let us show a proof, or at least sketch the ideas. First of all (ii) $\Rightarrow$ (iii) $\Rightarrow$ (iv) is obvious.

(i) $\Rightarrow$ (v): There exists $A\in s_d(n)$ which maps $y$ to $x$ so $A(y-td)=Ay-tAd=x-td$ for all $t\in\mathbb R$, hence (v) is a direct consequence of \eqref{eq:doubly_stoch_trace_norm}.

(v) $\Rightarrow$ (iv): Because $\unitvector^{\top}x=\unitvector^{\top}y$, just as in the proof of Lemma \ref{lemma_trace_norm_ball_maj} trace equality and trace norm inequality implies the inequality for the positive part of the vectors.

(iv) $\Rightarrow$ (ii): Let a continuous convex function $\psi:D(\psi)\subseteq\mathbb R\to\mathbb R$ be given such that $P:=\{ \frac{x_j}{d_j},\frac{y_j}{d_j} : j=1,\ldots,n \}\subseteq D(\psi)$. In particular one can construct a continuous function $\tilde\psi:\mathbb R\to\mathbb R$ such that
\begin{itemize}
\item[$\bullet$] $\tilde\psi(\frac{x_j}{d_j})=\psi(\frac{x_j}{d_j})$ and $\tilde\psi(\frac{y_j}{d_j})=\psi(\frac{y_j}{d_j})$ for all $j=1,\ldots,n$
\item[$\bullet$] $\tilde\psi$ is piecewise linear with change in slope only at the elements of $P$.
\item[$\bullet$] $\tilde\psi$ is convex (evident because $\psi$ is convex).
\end{itemize}
In other words $\tilde\psi$ is the ``piecewise linearization'' of $\psi$ (with respect to $P$). Thus it suffices to prove (ii) for all such $\tilde\psi$ because then
$$
\sum\nolimits_{j=1}^n d_j \psi\big(\frac{x_j}{d_j}\big)=\sum\nolimits_{j=1}^n d_j \tilde\psi\big(\frac{x_j}{d_j}\big)\leq \sum\nolimits_{j=1}^n d_j \tilde\psi\big(\frac{y_j}{d_j}\big)= \sum\nolimits_{j=1}^n d_j \psi\big(\frac{y_j}{d_j}\big)
$$

Now let $\phi:\mathbb R\to\mathbb R$ continuous, convex and piecewise linear (with respect to $P$) be given. Then $\phi$ can be written as a (non-negative) linear combination of the maps\footnote{
This is true up to an affine linear map which due to $\unitvector^{\top}x=\unitvector^{\top}y$ yields equality in (ii), thus can be disregarded.
} $\{\phi_p : p\in P\}$ where $\phi_p(t):=(p-t)_+$. But all $\phi_p$ satisfy (ii) by assumption, hence $\phi$ does as well.

(ii) $\Rightarrow$ (i): The idea here is much in the spirit of Kemperman \cite[Thm.~2]{Kemperman75}. Finding $A\in\mathbb R_+^{n\times n}$ with $\unitvector^{\top}A=\unitvector^{\top}$, $Ad=d$ and $Ay=x$ is equivalent (by vectorization, cf.~\cite[Ch.~2.4]{MN07}) to finding a solution $z\in\mathbb R_+^{n^2}$ to
$$
\begin{pmatrix} y^\top \otimes \identity_n\\d^\top \otimes\identity_n\\\identity_n\otimes \unitvector^{\top} \end{pmatrix}z=\begin{pmatrix} x\\d\\e \end{pmatrix}
$$
where $\otimes$ is the usual Kronecker product \cite[Ch.~2.2]{MN07} and $z=\operatorname{vec} A$. By Farkas' lemma\footnote{Farkas' lemma states that for $m,n\in\mathbb N$, $A\in\mathbb R^{m\times n}$, $b\in\mathbb R^m$, the system of linear equations $Ax=b$ has a solution in $\mathbb R_+^n$ if and only if for all $y\in\mathbb R^m$ which satisfy $A^\top y\leq 0$ one has $b^\top y\leq 0$, refer to \cite[Coro.~7.1.d]{Schrijver86} (when replacing $A,b$ by $-A,-b$).} such a solution exists if (and only if) for all $w\in\mathbb R^{3n}$ which satisfy
\begin{align}
\Big(\begin{pmatrix} y^\top \otimes \identity_n\\d^\top \otimes\identity_n\\\identity_n\otimes \unitvector^{\top} \end{pmatrix}^\top \begin{pmatrix} \vec{w_1}\\\vec{w_2}\\\vec{w_3} \end{pmatrix}\Big)_{n(j-1)+k}&=\Big(\begin{pmatrix} (y\otimes I_n)\vec{w_1}&(d\otimes I_n)\vec{w_2}&(I_n\otimes e)\vec{w_3}\end{pmatrix}\Big)_{n(j-1)+k}\notag\\
&=y_jw_k+d_jw_{n+k}+w_{2n+j}\leq 0\label{eq:w_y_ineq}
\end{align}
for all $j,k=1,\ldots,n$ one has
$$
\sum\nolimits_{j=1}^n (x_jw_j+d_jw_{n+j}+w_{2n+j})\leq 0\,.
$$
Consider the convex (because affine linear) functions $\psi_j:\mathbb R\to\mathbb R$, $t\mapsto w_jt+w_{n+j}$ for all $j=1,\ldots,n$. Then
$$
\psi:\mathbb R\to\mathbb R\qquad t\mapsto\max_{j=1,\ldots,n}\psi_j(t)
$$
is convex and continuous as well so by assumption and because $d>0$
\begin{align*}
\sum\nolimits_{j=1}^n (x_jw_j+d_jw_{n+j}+w_{2n+j})&=\sum\nolimits_{j=1}^n d_j\psi_j(\tfrac{x_j}{d_j})+w_{2n+j}\\
&\leq \sum\nolimits_{j=1}^n d_j\psi(\tfrac{x_j}{d_j})+w_{2n+j}\\
&\leq \sum\nolimits_{j=1}^n d_j\psi(\tfrac{y_j}{d_j})+w_{2n+j}\,.
\end{align*}
But now for every $j=1,\ldots,n$ exists $k=k(j)$ such that $\psi(\tfrac{y_j}{d_j})=\psi_{k(j)}(\tfrac{y_j}{d_j})$ by definition of $\psi$ (the maximum has to be attained by at least one of the $\psi_k$). Hence
\begin{align*}
\sum\nolimits_{j=1}^n (x_jw_j+d_jw_{n+j}+w_{2n+j})&\leq \sum\nolimits_{j=1}^n d_j\psi(\tfrac{y_j}{d_j})+w_{2n+j}\\
&= \sum\nolimits_{j=1}^n d_j\psi_{k(j)}(\tfrac{y_j}{d_j})+w_{2n+j}\\
&=\sum\nolimits_{j=1}^n \underbrace{y_jw_{k(j)}+d_jw_{n+k(j)}+w_{2n+j}}_{\leq 0\text{ by }\eqref{eq:w_y_ineq}}\leq 0
\end{align*}
so we are done. Note that we needed access to not all, but only to the piecewise linear convex functions---this is the same effect as in the proof of (iv) $\Rightarrow$ (ii).
%%ARXIV ONLY

(i) $\Leftrightarrow$ (vii): This will be a direct consequence of our considerations in Section \ref{sec_d_maj_poly} so we will postpone this part of the proof to Section \ref{sec_d_poly_analysis}. Note that we will not use this result anywhere in the paper, meaning we are not at risk to run into a circular argument.
\end{proof}

\begin{rmk}
\begin{itemize}
\item[(i)] 
    In terms of numerics, Prop.~\ref{lemma_char_d_vec} (vi) is the most efficient one for checking $d$-majorization as it encapsulates at most $n+1$ constraints one has to verify. In contrast, the definition of $d$-majorization, i.e.~finding a $d$-stochastic matrix which maps $y$ to $x$
    % -- while this can be done efficiently --
    boils down to a linear programming problem which is by far not as easy to check as condition (vi). The other conditions from Prop.~\ref{lemma_char_d_vec} consist of either uncountably many constraints (conditions (ii), (iii), and (v)), or at most $2n$ (condition (iv)) or $n^2$ steps (condition (vii), because each of the $n-1$ constraints features a minimum over $n$ numbers which one has to compute before checking the actual constraint).
%  \marginpar{sorting condition and $+$-operation!!!}
\item[(ii)] In the physics lecture, also connected to the notion of thermo-majorization, a freshly published result \cite{Shiraishi20} gives a direct proof of (vii) $\Rightarrow$ (i) from Prop.~\ref{lemma_char_d_vec}. Given two vectors $x,y\in\mathbb R^n$ of the same ``trace'' such that $y$ thermo-majorizes $x$, Shiraishi gave a constructive algorithm for a $d$-stochastic matrix which sends $y$ to $x$.
\end{itemize}
\end{rmk}

Recently Alhambra et al.~\cite{Alhambra19} were able to find conditions under which classical majorization implies $d$-majorization for $d$ from some parameter range%, provided the initial vector and the $d$-vector are ordered opposite to each other
. As their result was obtained in the context of dephasing thermalization let us reformulate it by casting it into our notation:
\begin{proposition}
The following statements hold.
\begin{itemize}
\item[(i)] Let $x,y\in\mathbb R^n$ and $d\in\mathbb R_{++}^n$ be given. If $x,y$ are similarly $d$-ordered, i.e.~there exists a permutation $\sigma\in S_n$ such that $\frac{x_{\sigma(1)}}{d_{\sigma(1)}}\geq\ldots\geq\frac{x_{\sigma(n)}}{d_{\sigma(n)}}$ and $\frac{y_{\sigma(1)}}{d_{\sigma(1)}}\geq\ldots\geq \frac{y_{\sigma(n)}}{d_{\sigma(n)}}$, then $x\prec_d y$ holds if and only if $x\prec y$.
\item[(ii)] Let $y\in\mathbb R_+^n$ and $d\in\mathbb R_{++}^n$. If $y_1\leq\ldots\leq y_n$ and $d_1\geq\ldots\geq d_n$, then $\underline{\sigma}y\prec_d y$ for all $\sigma\in S_n$.
\item[(iii)] Let $y\in\mathbb R_+^n$ with $y_1\leq\ldots\leq y_n$ be given. Then for all $x\in\mathbb R^n$ one has $x\prec y$ if and only if $x\prec_d y$ for all $d\in\mathbb R_{++}^n$ with $d_1\geq\ldots\geq d_n$.
\end{itemize}
\end{proposition}
\begin{proof}
(i): \cite[Coro.~2.5]{Joe90}. Note that $x\prec_d y$ if and only if $\underline{\sigma}x\prec_{\underline{\sigma}d}\underline{\sigma}y$ for all $\sigma\in S_n$ as is readily verified, so it suffices to have $x,y$ similarly $d$-ordered. (ii): One can explicitly write down generalized T-transforms which first shift $y_1$ to $y_{\sigma(1)}$, then $y_2$ to $y_{\sigma(2)}$, and so on. The details are carried out in \cite[p.~13 \& 14]{Alhambra19}. (iii), $\Leftarrow$: Obvious. (iii), $\Rightarrow$: Let $x\in\mathbb R^n$ with $x\prec y$ be given and let $\tau\in S_n$ be any permutation such that $\underline{\tau}x_1\leq\ldots\leq \underline{\tau}x_n$. Then (ii) implies $\underline{\sigma}\,\underline{\tau}x\prec_d\underline{\tau}x$ for all $\sigma\in S_n$ and all $d\in\mathbb R_{++}^n$ with $d_1\geq\ldots\geq d_n$, so choosing $\sigma=\tau^{-1}$ yields $x\prec_d\underline{\tau}x$. On the other hand $\underline{\tau}x$ and $y$ are similarly $d$-ordered for all such $d$---because $\frac{(\underline{\tau}x)_1}{d_1}\leq\ldots\leq\frac{(\underline{\tau}x)_n}{d_n}$ and $\frac{y_1}{d_1}\leq\ldots\leq \frac{y_n}{d_n}$---so we have $\underline{\tau}x\prec_d y$ by (i). Using that $\prec_d$ is a preorder this yields $x\prec_d\underline{\tau}x\prec_dy$, that is, $x\prec_d y$ as claimed.
\end{proof}

To conclude this section we make some statements about minimal and maximal elements of the preorder $\prec_d$.
\begin{thm}\label{prop_1}
Let $d\in\mathbb R_{++}^n$ be given. The following statements hold.
\begin{itemize}
\item[(i)] $d$ is the unique minimal element within $\lbrace x\in\mathbb R^n : \unitvector^{\top}x=\unitvector^{\top}d\rbrace$ %(i.e.~the trace hyperplane ``spanned'' by $d$) 
with respect to $\prec_d\,$.
\item[(ii)] Let $j,k\in\{1,\ldots,n\}$, $j\neq k$ be given. Then $e_j\prec_d e_k$ if and only if $d_j\geq d_k$.
\item[(iii)] If $k$ is chosen such that $d_k$ is minimal in $d$, then $(\unitvector^{\top}d) e_k$ is maximal in $(\unitvector^{\top}d)\Delta^{n-1}=\lbrace x\in\mathbb R_+^n : \unitvector^{\top}x=\unitvector^{\top}d\rbrace$ with respect to $\prec_d$. It is the unique maximal element in $(\unitvector^{\top}d)\Delta^{n-1}$ with respect to $\prec_d$ if and only if $d_k$ is the unique minimal element of $d$.
\end{itemize}
\end{thm}
\begin{proof}
(i) Consider $d\unitvector^{\top}/(\unitvector^{\top}d)\in s_d(n)$ which maps any $x\in\mathbb R^n$ with $\unitvector^{\top}x=\unitvector^{\top}d$ to $d$ so $d\prec_d x$. Uniqueness is obvious as $d$ is a fixed point of every $d$-stochastic matrix.\smallskip

(ii): Applying Prop.~\ref{lemma_char_d_vec} (vi), $e_j\prec_d e_k$ is equivalent to $\|e_j-\frac{1}{d_k}d\|_1\leq \|e_k-\frac{1}{d_k}d\|_1$ (as the $n-1$ other inequalities read $1\leq 1$ and thus are redundant). But this is satisfied iff
$$
\Big|1-\frac{d_j}{d_k}\Big|+\frac{d_k}{d_k}\leq \frac{d_j}{d_k}+\Big|1-\frac{d_k}{d_k}\Big|\quad\Leftrightarrow\quad \Big|\frac{d_j}{d_k}-1\Big|\leq \frac{d_j}{d_k}-1
$$
which obviously holds iff $\frac{d_j}{d_k}-1\geq 0$, that is, $d_j\geq d_k$.\smallskip

(iii): W.l.o.g.~$\unitvector^{\top}d=1$. Because $d_k\leq d_j$ for all $j=1,\ldots,n$, (ii) implies $e_j\prec_d e_k$ which, using convexity of $\prec_d$, shows maximality of $e_k$. Moreover, if $d_k$ is not the \textit{unique} minimal element---but there exists $i\neq k$ such that $d_i=d_k$---then by the same argument $e_i$ is maximal in $\Delta^{n-1}$ w.r.t.~$\prec_d$, so there exist at least two maximal elements.

Thus all that is left to show is that if $d_k$ is the unique minimal element of $d$, then $e_k\prec_d x$ for any $x\in\Delta^{n-1}$ implies $x=e_k$ meaning there cannot be any maximal element other than $e_k$. Indeed $e_k\prec_d x$ by Prop.~\ref{lemma_char_d_vec} (vi) is equivalent to $\|e_k-\frac{x_i}{d_i}d\|_1\leq\|x-\frac{x_i}{d_i}d\|_1$ for all $i=1,\ldots,n$ which after a straightforward computation reads
\begin{align*}
1+\frac{x_i}{d_i}(1-d_i-d_k)&\leq 1-2\sum_{\{\alpha:\frac{x_\alpha}{d_\alpha}<\frac{x_i}{d_i}\}}x_\alpha +\frac{x_i}{d_i}\Big(1-2d_i-2\sum_{\{\alpha:\frac{x_\alpha}{d_\alpha}>\frac{x_i}{d_i}\}}d_\alpha\Big)\\
&\leq 1 +\frac{x_i}{d_i}(1-2d_i)\leq 1+\frac{x_i}{d_i}(1-d_i-d_k)\,.
\end{align*}
Hence all these inequalities are actually equalities; in particular the last step then implies $x_i=0$ for all $i\neq k$ because $d_i>d_k$ by assumption.
\end{proof}
\begin{rmk}\label{rem_pure_state_transform}
The fact that every $e_1,\ldots,e_n$ is maximal in the standard simplex $\Delta^{n-1}$ for $d=\unitvector$ is lost in the general setting (consider the example from Remark \ref{rem_maj_vector} (iv)).

However, for strictly positive vectors $z\in\mathbb R_{++}^n$ one still has $(\unitvector^{\top}z )e_k\not\prec_d z$ for all $k=1,\ldots,n$. More generally, if $y\prec_d z$ then $y$ has to be strictly positive as well; otherwise the corresponding transformation matrix (non-negative entries) would contain a row of zeros which---due to $d>0$---contradicts $d$ being one of its fixed points. This is a special case of strict positivity of matrix-$D$-majorization \cite[Coro.~4.7]{vomEnde20Dmaj}
\end{rmk}
\section{Descriptions \& Properties of Convex Polytopes}\label{sec_prelim_vector}

Before we investigate what Proposition \ref{lemma_char_d_vec} tells us about sets of the form $\{x\in\mathbb R^n:x\prec_dy\}$ for some $y\in\mathbb R^n$, $d\in\mathbb R_{++}^n$ we need some basic knowledge of convex polytopes.
Usually, convex polytopes are introduced as subsets of $\mathbb R^n$ which can be written as the convex hull of finitely many vectors from $\mathbb R^n$, cf.~\cite[Ch.~7.2]{Schrijver86}, \cite[Ch.~3]{Gruenbaum03}. Now it is well known that such polytopes can be characterized via finitely many affine half-spaces: More precisely a set $P\subset \mathbb R^n$ is a convex polytope if and only if $P$ is bounded and there exist $m\in\mathbb N$, $A\in\mathbb R^{m\times n}$, and $b\in\mathbb R^m$ such that $P=\{x\in\mathbb R^n : Ax\leq b\}$ \cite[Coro.~7.1c]{Schrijver86}. These characterizations of convex polytopes are also known as vertex- and halfspace-description ($\mathscr V$- and $\mathscr H$-description), respectively \cite[Ch.~3.6]{Gruenbaum03}. 
\begin{rmk}\label{rem:halfspace_actions}
Let any $A\in\mathbb R^{m\times n}$, $b,b'\in\mathbb R^m$, and $p\in\mathbb R^n$ be given. The following observations are readily verified:
\begin{align*}
\{x\in\mathbb R^n : Ax\leq b\}\cap \{x\in\mathbb R^n : Ax\leq b'\}&=\{x\in\mathbb R^n : Ax\leq\min\{b,b'\}\}\\
\{x\in\mathbb R^n : Ax\leq b\}+p&=\{x\in\mathbb R^n : Ax\leq b+Ap\}\\
b\leq b'\quad\Leftrightarrow\quad\{x\in\mathbb R^n : Ax\leq b\}&\subseteq \{x\in\mathbb R^n : Ax\leq b'\}\,.
\end{align*}
Here and henceforth, we for simplicity use the convention that $\min$ and $\max$ operate entrywise on vectors, e.g., $\min\{b,b'\}=(\min\{b_j,b_j'\})_{j=1}^m$ for $b,b'\in\mathbb R^m$.
The above results are not too surprising as $A$ in some sense describes the geometry of the polytope which intuitively should not change under the above operations.
\end{rmk}
In the following we are primarily interested in the case where $A$ and $b$ have a very particular structure. After all, in the introduction we have seen that the orientation of the half-spaces which describe classical majorization are precisely given by the collection of all binary vectors $\{0,1\}^n$.
Thus to introduce
 this special structure we define
% \footnote{Note that $\sum_{j=1}^{n-1}\binom{n}{j}+2=\sum_{j=0}^n\binom{n}{j}=2^n$ by the binomial theorem, which shows $M\in\mathbb R^{2^n\times n}$.}
\begin{equation}\label{eq:M_maj}
 M:={\footnotesize\begin{pmatrix} M_1\\M_2\\\vdots\\M_{n-1}\\\unitvector^{\top}\\-\unitvector^{\top} \end{pmatrix}}
 \in \mathbb R^{2^n\times n}\,,
\end{equation}
where the rows of $M_j\in\mathbb R^{\binom{n}{j}\times n}$
% are made up of all elements of 
% $$%\label{eq:S_j}
% \Big\{x\in\{0,1\}^n : \sum\nolimits_{i=1}^n x_i=j\Big\}
% $$
 consist of all elements of $\{0,1\}^n$ which sum up to $j$. The order of these rows can be chosen arbitrarily but shall be
 fixed henceforth.
 
 Moreover, $b\in\mathbb R^{2^n}$ will be partitioned in the same way and our focus will be on the case where
 $b'_{n}+b'_{n+1}=0$ or, equivalently,
$$%\label{eq:b'_canon}
b={\footnotesize \begin{pmatrix} b'_1\\\vdots\\b'_{n-1}\\b'_n\\-b'_n\end{pmatrix} }\in\mathbb R^{2^n}\,.
$$
with $b'_j\in\mathbb R^{\binom{n}{j}}$ for all $j=1,\ldots,n$; in particular $b'_n\in\mathbb R$. The last two conditions on $M$ and $b$
 obviously allow us to guarantee that any solution of $Mx \leq b$ satisfies the ``trace condition'' $\unitvector^{\top}x=\sum_{j=1}^nx_j = b'_n$.
% Here and henceforth \textbf{any vector} $b\in\mathbb R^{2^n}$ shall automatically be of the form \eqref{eq:b'_canon}

%This form of $b$ already has consequences for the set of solutions to $Mx\leq b$.
\begin{lem}\label{lemma_Mx_b_convex_polytope}
 Let $M$ be the matrix \eqref{eq:M_maj} and let $b\in\mathbb R^{2^n}$ with $b'_{n}+b'_{n+1}=0$ be given.
 If $\{x\in\mathbb R^n : Mx\leq b\}$ is non-empty then it is a convex polytope of dimension at most $n-1$.
\end{lem}
\begin{proof}
Because $b'_n +b'_{n+1}=0$ by assumption, all solutions to $Mx\leq b$ have to satisfy $\unitvector^{\top}x=b'_n $ which reduces the dimension of $\{x\in\mathbb R^n : Mx\leq b\}$ by at least $1$. Now if we can show that $\{x\in\mathbb R^n : Mx\leq 0\}=\{0\}$ then the set in question is bounded (cf.~\cite[Ch.~8.2]{Schrijver86}) so by the above characterization of convex polytopes we are done. Indeed let $x\in\mathbb R^n$ satisfy $Mx\leq 0$. Then $M_1x\leq 0$, implying $x_i\leq 0$ for all $i=1,\ldots,n$. But this together with the ``trace condition'' $\unitvector^{\top}x=0$ shows $x_i=0$ for all $i$, as desired. 
\end{proof}

Now an immediate question is the one concerning the vertices (extreme points) of convex polytopes of the above form. This will be the topic of the remainder of this section.
For this we need a characterization of the extreme points, given a
$\mathscr H$-description of a convex polytope.

To characterize these extreme points in a convenient way we need a formalism to pick rows
 from the matrix $M$ and the corresponding entry of the vector $b$. Since we do not want to fix any ``unnatural''
 order of the rows of $M$, we define an abstract mapping which does this job for us. 
   While this procedure would not be necessary for classical majorization -- because there the vector $b$ is of convenient
   structure -- for studying general $d$-majorization this map will be indispensable.
\begin{defi}\label{def_mathfrak_b}
Let $a\in\{0,1\}^n$, $a\neq 0$ be given. Then the row vector $a$ corresponds to a (unique) row of $M\in\mathbb R^{2^n\times n}$, hence there exist unique $m_1\in\{1,\ldots,n\}$, $m_2\in\{1,\ldots,\binom{n}{m_1}\}$ such that
%$
%p_k=(M_{m_1})_{m_2k}
%$
%for all $k=1,\ldots,n$, i.e.~
$a$ is the $m_2$-th row of the submatrix $M_{m_1}$ of $M$ in \eqref{eq:M_maj}. This lets us define
\begin{align*}
\mathfrak m:\{0,1\}^n&\to\mathbb N_0\times\mathbb N_0\\
a&\mapsto\begin{cases}
(0,0)&\text{if }a=0\\
(m_1,m_2)&\text{else}
\end{cases}\,.
\end{align*}
Now given $b\in\mathbb R^{2^n}$ with $b'_n +b'_{n+1}=0$ we can define an analogous mapping
\begin{align*}
\mathfrak b:\{0,1\}^n&\to\mathbb R\\
a&\mapsto\begin{cases}
0&\text{if }a=0\\
(b'_{m_1})_{m_2}&\text{else}
\end{cases}
\end{align*}
where $(m_1,m_2)=\mathfrak m(a)$.
\end{defi}
In other words $Mx\leq b$ then is equivalent to $a^\top x\leq\mathfrak b(a^\top )$ for all $a\in\{0,1\}^n$ together with $\unitvector^\top x=b_n'$. Certainly the maps $\mathfrak m,\mathfrak b$ generalize to matrices $A\in\{0,1\}^{m\times n}$, $b\in\mathbb R^m$ ($m\in\mathbb N$) by applying them to every row of $A$ individually\footnote{For this map to be well-defined we need the assumption that no row of $A$ appears twice. But this is rather natural because one of the two inequalities $a^\top x\leq b$, $a^\top x\leq b'$ is redundant and can be disregarded.}.
%$$
%\mathfrak m(A)=\mathfrak b\Big( \begin{pmatrix} a_1\\\vdots\\a_m \end{pmatrix} \Big)=\begin{pmatrix} \mathfrak b(a_1)\\\vdots\\\mathfrak b(a_m) \end{pmatrix}\quad \mathfrak b(A)=\mathfrak b\Big( \begin{pmatrix} a_1\\\vdots\\a_m \end{pmatrix} \Big)=\begin{pmatrix} \mathfrak b(a_1)\\\vdots\\\mathfrak b(a_m) \end{pmatrix}\,.
%$$
\begin{lem}\label{lemma_extreme_points_h_descr}
Let $b\in\mathbb R^{2^n}$ with $b'_n +b'_{n+1}=0$ as well as $p\in\mathbb R^n$ be given such that $Mp\leq b$. The following statements are equivalent.
\begin{itemize}
\item[(i)] $p$ is an extreme point of $\{x\in\mathbb R^n : Mx\leq b\}$.
\item[(ii)] There exists a submatrix $M'\in\mathbb \{0,1\}^{n\times n}$ of $M$---one row of $M'$ being equal to $\unitvector^{\top}$---such that $M'p=\mathfrak b(M')=:b'$ and $\operatorname{rank}M'=n$.
\end{itemize}
%Furthermore, if $M E_{b}(\sigma)\leq b$ then $E_{b}(\sigma)$ is an extreme point of $\{ x\in\mathbb R^n : Mx\leq b\}$. 
\end{lem}
\begin{proof}
%If we can prove (i) $\Leftrightarrow$ (ii) then the additional statement is a directly implied using Lemma \ref{lemma_E_b_sigma_properties} (and the definition of the function $f$).
%
``(ii) $\Rightarrow$ (i)'': Assume there exist $x_1,x_2\in\{x\in\mathbb R^n : Mx\leq b\}$ and $\lambda\in (0,1)$ such that $p=\lambda x_1+(1-\lambda)x_2$. Then $M'x_1\leq b'$, $M'x_2\leq b'$ by assumption and thus
$$
b'=M'p=\lambda M'x_1+(1-\lambda)M'x_2\leq \lambda b'+(1-\lambda)b'=b'\quad\Rightarrow\quad M'x_1=M'x_2=b'\,.
$$
But $M'\in\mathbb R^{n\times n}$ is of full rank so the system of linear equations $M'y=b$ has a unique solution in $\mathbb R^n$, hence $x_1=x_2=p$ is in fact an extreme point of $\{x\in\mathbb R^n : Mx\leq b\}$.

``(i) $\Rightarrow$ (ii)'': %By Lemma \ref{lemma_Mx_b_convex_polytope} the set $\{x\in\mathbb R^n : Mx\leq b\}$ is a non-empty, $n-1$-dimensional convex polytope. Thus we may 
%apply \cite[Thm.~3.1.7]{Gruenbaum03} with $d=n-1$, $h=0$, and $k=n-2$ which shows that---aside from the trace condition---$p$ has to be a solution to $\hat My=\hat b$ with $\hat M\in\mathbb R^{(n-1)\times n}$, $b\in\mathbb R^{n-1}$, and $\operatorname{rank}\hat M=n-1$. 
Each extreme point $p$ of $\{x\in\mathbb R^n : Mx\leq b\}$ is determined by $n$ linearly independent equations from $Mx=b$ so there exists a submatrix $\hat M\in\mathbb R^{n\times n}$ of $M$ of full rank such that $\hat Mp=\mathfrak b(\hat M)=:\hat b$, cf.~\cite[Thm.~8.4 ff.]{Schrijver86}. If one row of $\hat M$ equals $\unitvector^{\top}$ then we are done. If $\hat M$ features $-\unitvector^{\top}$ replace it by $\unitvector^{\top}$ and flip the sign of the corresponding entry in $\mathfrak b(\hat M)$. Otherwise define
$$
\tilde{M}:={\begin{pmatrix} \hat M\\\unitvector^{\top} \end{pmatrix}}\in\mathbb R^{(n+1)\times n}\quad\text{ and }\quad\tilde{b}:=\mathfrak b(\tilde M)={\begin{pmatrix} \hat b\\b'_n \end{pmatrix}}\in\mathbb R^{n+1}
$$
so $\tilde{M}p=\tilde{b}$ because $p$ satisfies the ``trace condition'' $\unitvector^{\top}p=b'_n$. But this system of linear equations is now overdetermined so there exists a row of $\hat{M}$ which can be replaced by $\unitvector^{\top}$ such that the resulting matrix $M'\in\mathbb \{0,1\}^{n\times n}$---which still satisfies $M'p=\mathfrak b(M')$---has again full rank.
%, contains $\unitvector^{\top}$, and %together with the corresponding subvector $b'\in\mathbb R^n$ of $b$ 
%satisfies $M'p=\mathfrak b(M')$.
%because $\hat M$ is of full rank it can via elementary row operations be transformed into a submatrix $M'$ of $M$---the former still of full rank---such that one of the rows of $M'$ equals $\unitvector^{\top}$ and the corresponding entry of $b'$ equals $\beta$ (so $b'$ is still the corresponding subvector of $b$).
\end{proof}

This enables---in some special cases---an explicit description of the extreme points of the polytope induced by $M$ and $b$.
Henceforth the symbol $\Lambda$ will denote the lower triangular matrix
\begin{equation}\label{eq:Lambda}
\Lambda={\footnotesize\begin{pmatrix}
1&0&\cdots&0\\
\vdots&\ddots&\ddots&\vdots\\
\vdots&&\ddots&0\\
1&\cdots&\cdots&1
\end{pmatrix}}\,.
\end{equation}

\begin{defi}\label{def_Ebsigma}
Let $b\in\mathbb R^{2^n}$ with $b'_n +b'_{n+1}=0$ and arbitrary $\sigma\in S_n$ be given. Denote by $\underline{\sigma}$ the permutation matrix\footnote{Given some permutation $\sigma\in S_n$ the corresponding permutation matrix is 
given by $\sum_{i=1}^n e_i e_{\sigma(i)}^\top $. In particular the identities $(\underline{\sigma}x)_j=x_{\sigma(j)}$ and $\underline{\sigma\circ\tau}=\underline{\tau}\cdot\underline{\sigma}$ hold.\label{footnote_permutation_matrix}} induced by $\sigma$. Then the unique solution to
\begin{equation}\label{eq:E_sigma_unique_sol}
\Lambda\,\underline{\sigma}x=\mathfrak b\big( \Lambda\,\underline{\sigma} \big)=:b'_\sigma
\end{equation}
(with $\Lambda$ from \eqref{eq:Lambda}) shall be denoted by $x=E_b(\sigma)$.
\end{defi} 

Now $E_b(\sigma)$ is of the following simple form.
\begin{lem}\label{lemma_E_b_sigma_properties}
Let $b\in\mathbb R^{2^n}$ with $b'_n +b'_{n+1}=0$, arbitrary $\sigma\in S_n$, as well as $p\in\mathbb R^n$ be given. 
Then for all $j=1,\ldots,n$ and for all $\sigma\in S_n$
\begin{equation}\label{eq:E_sigma_property}
(E_b(\sigma))_{\sigma(j)}=\mathfrak b\big(\sum\nolimits_{i=1}^j e_{\sigma(i)}^\top \big)-\mathfrak b\big(\sum\nolimits_{i=1}^{j-1} e_{\sigma(i)}^\top \big)=(b'_\sigma)_{j}-(b'_\sigma)_{j-1}
\end{equation}
(with $b'_\sigma$ from \eqref{eq:E_sigma_unique_sol}), as well as $E_{b+Mp}(\sigma)=E_b(\sigma)+p$.
\end{lem}
\begin{proof}
%If we can show that $E_{b}(\sigma)$ solves \eqref{eq:E_sigma_unique_sol} then it has to be the unique solution as the matrix on the l.h.s.~is clearly of full rank. Now 
The $j$-th row of \eqref{eq:E_sigma_unique_sol} for $x=E_{b}(\sigma)$, $j=1,\ldots,n$ reads
\begin{align*}
\sum\nolimits_{i=1}^j (E_{b}(\sigma))_{\sigma(i)}=\big(\sum\nolimits_{i=1}^j e_{\sigma(i)}^\top \big) E_{b}(\sigma)=\mathfrak b\big(\sum\nolimits_{i=1}^j e_{\sigma(i)}^\top \big)
%=\mathfrak b\big(\big(\sum\nolimits_{i=1}^j e_i^\top \big)\underline{\sigma}\big)
\end{align*}
which implies \eqref{eq:E_sigma_property}. Also one readily verifies
$$
(b+Mp)'_\sigma=b'_\sigma+\underbrace{\big(\sum\nolimits_{i=1}^j p_{\sigma(i)}\big)_{j=1}^n}_{\text{corresponding entry in }Mp}=b'_\sigma+\Lambda\,\underline{\sigma}p
$$
for any $\sigma\in S_n$, so $E_{b+Mp}(\sigma)=E_b(\sigma)+p$ by uniqueness of the solution of \eqref{eq:E_sigma_unique_sol}. 
\end{proof}

Clearly if $E_b(\sigma)\in\{x\in\mathbb R^n : Mx\leq b\}$ for some $\sigma\in S_n$, then it is an extreme point by Lemma \ref{lemma_extreme_points_h_descr} although, in general, not every $E_b(\sigma)$ needs to be in $\{x\in\mathbb R^n : Mx\leq b\}$ for arbitrary $b\in\mathbb R^{2^n}$ with $b'_n +b'_{n+1}=0$ (cf.~Example \ref{ex_proof_ext_point_fail}, \ref{app_b}). However for the well-structured polytopes we will deal with later on all of these $E_b(\sigma)$ lie within the polytope, in which case they are the only extreme points:

\begin{thm}\label{thm_Ebsigma_extreme}
Let $b\in\mathbb R^{2^n}$ with $b'_n +b'_{n+1}=0$ be given such that $\{E_b(\sigma) : \sigma\in S_n\}\subset \{x\in\mathbb R^n : Mx\leq b\}$. Then every extreme point of $\{x\in\mathbb R^n : Mx\leq b\}$ is of the form $E_b(\sigma)$ for some $\sigma\in S_n$, and therefore $\{x\in\mathbb R^n : Mx\leq b\}=\operatorname{conv}\{E_b(\sigma) : \sigma\in S_n\}$.
\end{thm}
\begin{proof}
Let $p\in \{x\in\mathbb R^n : Mx\leq b\}$ be extremal so by Lemma \ref{lemma_extreme_points_h_descr} there exists a submatrix $M'\in\{0,1\}^{n\times n}$ of $M$ of full rank, one row of $M'$ being equal to $\unitvector^{\top}$, such that $M'p=\mathfrak b(M')$. By Minkowski's theorem \cite[Thm.~5.10]{Brondsted83} if we can show that $p=E_b(\sigma)$ for some $\sigma\in S_n$ this would conclude the proof.

Indeed consider any two rows $m_1,m_2$ of $M'$. The idea will be to show that
$m_\mathrm{min}p=\mathfrak b(m_\mathrm{min})$ and $m_\mathrm{max}p=\mathfrak b(m_\mathrm{max})$ where $m_\mathrm{min}:=\min\{m_1,m_2\}$, $m_\mathrm{max}:=\max\{m_1,m_2\}$. Thus $M'$ can be extended by these rows while keeping the equality $M'p=\mathfrak b(M')$. This would conclude the proof by means of an abstract result which, roughly speaking, states that any matrix of full rank which has $\unitvector^{\top}$ as a row and satisfies this min-max-property necessarily features a submatrix of the form $\Lambda\,\underline{\sigma}$ for some $\sigma\in S_n$ (Lemma \ref{lemma_minmax_matrix} (iii), \ref{app_proofs}), so $p=E_b(\sigma)$.

Now note that 
%\begin{itemize}
%\item $m_\mathrm{min}\unitvector< m_1\unitvector\leq m_2\unitvector<m_\mathrm{max}\unitvector\,$.
%\item
$m_\mathrm{min}\leq m_\mathrm{max}$, meaning one finds $\tau\in S_n$ such that $m_\mathrm{min}$, $ m_\mathrm{max}$ are rows of
$
\Lambda\,\underline{\tau}
$ \footnote{
If $m_\mathrm{min}=0$ then one can trivially find $\tau\in S_n$ such that $ m_\mathrm{max}$ is a row of $\Lambda\,\underline{\tau}$.
}, as well as $m_\mathrm{min}+m_\mathrm{max}=m_1+m_2$.
%\end{itemize}
This has two immediate consequences: firstly, $M'p=\mathfrak b(M')$ and $Mp\leq b$ imply
\begin{align*}
\mathfrak b(m_1)+\mathfrak b(m_2)=(m_1+m_2)p=(m_\mathrm{min}+ m_\mathrm{max})p\leq \mathfrak b( m_\mathrm{min}) +\mathfrak b(m_\mathrm{max})\,,
\end{align*}
and secondly, $\Lambda\,\underline{\tau} E_b(\tau)=b'_\tau$ and $ME_b(\tau)\leq b$ (because $E_b(\tau)\in \{x\in\mathbb R^n : Mx\leq b\}$ by assumption) yield
\begin{align*}
\mathfrak b( m_\mathrm{min}) +\mathfrak b(m_\mathrm{max})=(m_\mathrm{min}+ m_\mathrm{max})E_b(\tau)=(m_1+m_2)E_b(\tau)\leq \mathfrak b(m_1)+\mathfrak b(m_2)\,.
\end{align*}
Combining these two we get
$$
\mathfrak b(m_1)+\mathfrak b(m_2)=(m_\mathrm{min}+ m_\mathrm{max})p\leq \mathfrak b( m_\mathrm{min}) +\mathfrak b(m_\mathrm{max})\leq \mathfrak b(m_1)+\mathfrak b(m_2)\,,
$$
that is, $(m_\mathrm{min}+ m_\mathrm{max})p= \mathfrak b( m_\mathrm{min}) +\mathfrak b(m_\mathrm{max})$. But $m_\mathrm{min}p\leq \mathfrak b( m_\mathrm{min})$, $ m_\mathrm{max}p\leq \mathfrak b(m_\mathrm{max})$ (due to $Mp\leq b$), hence
%\footnote{
%Given $a,b,c,d\in\mathbb R$, if $a\leq c$, $b\leq d$, and $a+b=c+d$ then $a=c+(d-b)\geq c\geq a$ so $a=c$ and thus $b=d$.
%}
$
m_\mathrm{min}p= \mathfrak b( m_\mathrm{min})$ and $m_\mathrm{max}p= \mathfrak b(m_\mathrm{max})
$.
As stated before, this means we can extend $M'$ by $m_\mathrm{min},m_\mathrm{max}$---assuming these were not part of $M'$ in the first place---to a matrix $M''$ which still satisfies $M''p=\mathfrak b(M'')$.
%This is the key to finishing this proof as $p$ not only satisfies $M'p=\mathfrak b(M')$ but even
%$$
%\begin{pmatrix}
%M'\\m_\mathrm{min}\\m_\mathrm{max}
%\end{pmatrix}p=\mathfrak b
%\begin{pmatrix}
%M'\\m_\mathrm{min}\\m_\mathrm{max}
%\end{pmatrix}.
%$$
%Thus taking \textit{any} two rows of $M'$, we may extend the matrix by their entrywise minimum and maximum, and $p$ still satisfies $M'p=\mathfrak b(M')$ (now for the possibly enlarged $M'$). Of course if $m_\mathrm{min}$ (or $m_\mathrm{max}$) was already a row of $M'$ then we need not add it to $M'$.
Repeating this enlargement process over and over will terminate eventually as $\{0,1\}^n$ is finite, yielding $\tilde M$
%: $M'$ can only grow but the set of possible rows is upper bounded by $\{0,1\}^n$, i.e.~by something finite. The final matrix $M'$ then is 
of full rank, containing $\unitvector^{\top}$, and, most importantly, for any two rows $m_1$, $m_2$ of $\tilde M$, $\min\{m_1,m_2\}$ and $\max\{m_1,m_2\}$ are rows of $\tilde M$ as well. Therefore Lemma \ref{lemma_minmax_matrix} (iii) (\ref{app_proofs}) yields a permutation $\sigma\in S_n$ such that every row of $ \Lambda\,\underline{\sigma} $ is a row of $\tilde M$, so $\tilde Mp=\mathfrak b(\tilde M)$ implies $\Lambda\,\underline{\sigma}p=\mathfrak b(\Lambda\,\underline{\sigma})$, that is, $p=E_b(\sigma)$.
\end{proof}
\begin{rmk}
This result is remarkable because it shows that under certain requirements on $b$ the special structure of $M$ allows to simplify the procedure of finding the extreme points of the induced polytope significantly. Usually one would have to determine all invertible submatrices of $M$, solve the corresponding linear equations
(cf.~\cite[Thm.~8.4 ff.]{Schrijver86}),
and finally check whether these solutions satisfy the remaining inequalities. However, $M$ has way more invertible submatrices than actual extreme points (i.e.~$n!$) in this case.
\end{rmk}

Following Remark \ref{rem_extend_lemma_minmax} (\ref{app_proofs}) one can even improve Thm.~\ref{thm_Ebsigma_extreme}: If the matrix $M'$ corresponding to $p$ contains two incomparable rows $m_1,m_2$, that is, $m_1\not\geq m_2\not\geq m_1$, then there exist at least two permutations $\sigma_1,\sigma_2\in S_n$ such that $E_b(\sigma_1)=p=E_b(\sigma_2)$. Notably in such a situation the map $\sigma\mapsto E_b(\sigma)$ is not injective.

\section{The $d$-Majorization Polytope}\label{sec_d_maj_poly}
\subsection{Characterizing the $d$-Majorization Polytope}\label{subsec_d_maj_poly}
With the tools surrounding convex polytopes developed in Section \ref{sec_prelim_vector} we are finally ready to explore the ``geometry'' of $d$-majorization. For this let us consider the set of all vectors which are $d$-majorized by some $y\in\mathbb R^n$; more generally we introduce the map
\begin{align*}
M_d:\mathcal P(\mathbb R^n)&\to\mathcal P(\mathbb R^n)\\
S&\mapsto \bigcup\nolimits_{y\in S}\lbrace x\in\mathbb R^n : x\prec_d y\rbrace
\end{align*}
where $\mathcal P$ denotes the power set. 
For convenience $M_d(y):=M_d(\{y\})$ for any $y\in\mathbb R^n$, which then equals the set of all vectors which are $d$-majorized by $y$. Note that the idea here is close to---but should not be confused with---the ($d$-)majorization polytope of two vectors \cite{Dahl99b} which, given two real vectors, is the set of all ($d$-)stochastic matrices which map one vector to the other.
\begin{lem}\label{lemma_closure_op}
Let $d\in\mathbb R_{++}^n$. Then $M_d$ is a closure operator\footnote{
Recall that an operator $J$ on the power set $\mathcal P(S)$ of a set $S$ is called \textit{closure operator} or \textit{hull operator} if it is extensive ($X\subseteq J(X)$), increasing ($X\subseteq Y\,\Rightarrow\,J(X)\subseteq J(Y)$) and idempotent ($J(J(X))=J(X)$) for all $X,Y\in\mathcal P(S)$, cf., e.g., \cite[p.~42]{Cohn81}.
}. In particular, for any $x,y\in\mathbb R^n$ one has $x\prec_dy$ if and only if $M_d(x)\subseteq M_d(y)$.
\end{lem}
\begin{proof}
The first statement is a simple consequence of the $d$-stochastic matrices $s_d(n)$ forming a semigroup with identity. For the second statement note that $x\prec_dy$, that is, $x\in M_d(y)$ implies $M_d(x)\subseteq M_d(M_d(y))=M_d(y)$.
\end{proof}
%\begin{proof}
%Obviously, $M_d$ is extensive and increasing. For idempotence ($M_d\circ M_d=M_d$), ``$\,\subseteq\,$'' follows from $Q_D(n)$ forming a semigroup and ``$\,\supseteq\,$'' is due to $\identity_{n\times n}\in Q_D(n)$.
%\end{proof}
Now Prop.~\ref{lemma_char_d_vec} (vi) directly implies
\begin{align}
M_d(y)&=\bigcap\nolimits_{i=1}^n\Big\lbrace x\in\mathbb R^n : \unitvector^{\top}x=\unitvector^{\top}y\ \wedge\ \Big\|x-\frac{y_i}{d_i}d\Big\|_1\leq\Big\|y-\frac{y_i}{d_i}d\Big\|_1\Big\rbrace\notag\\
&=\bigcap\nolimits_{i=1}^n\Big\lbrace x\in\mathbb R^n : \unitvector^{\top}\Big(x-\frac{y_i}{d_i}d\Big)=\unitvector^{\top}\Big(y-\frac{y_i}{d_i}d\Big)\ \wedge\ \Big\|x-\frac{y_i}{d_i}d\Big\|_1\leq\Big\|y-\frac{y_i}{d_i}d\Big\|_1\Big\rbrace\notag\\
&=\bigcap\nolimits_{i=1}^n\Big(\Big\lbrace \tilde x\in\mathbb R^n : \unitvector^{\top}\tilde x=\unitvector^{\top}\Big(y-\frac{y_i}{d_i}d\Big)\ \wedge\ \|\tilde x\|_1\leq\Big\|y-\frac{y_i}{d_i}d\Big\|_1\Big\rbrace+\frac{y_i}{d_i}d\Big)\label{eq:decomp_m_d}
\end{align}
for all $y\in\mathbb R^n$, $d\in\mathbb R_{++}^n$. 

\begin{lem}\label{lemma_trace_norm_ball_maj}
Let $z\in\mathbb R^n$%and let $\prec$ denote usual majorization, i.e.~$\prec_d$ for $d=\unitvector$
. Then
\begin{align*}
%&
\{x\in\mathbb R^n : \unitvector^{\top}x=\unitvector^{\top}z\ \wedge\ \|x\|_1\leq\|z\|_1\rbrace=\big\lbrace x\in\mathbb R^n\,\big|\,x\prec{\footnotesize(\unitvector^{\top}z_+,-\unitvector^{\top}z_-,0,\ldots,0 )^\top }\big\}%\\
%&=\bigcap\nolimits_{j=1}^{n-2}\Big\{(\unitvector^{\top}z_++(n-1)\unitvector^{\top}z_-)p-\unitvector^{\top}z_-e : p\in\Delta^{n-1},\sum\nolimits_{l=1}^j p_{[l]}\leq\tfrac{\unitvector^{\top}z_++j\unitvector^{\top}z_-}{\unitvector^{\top}z_++(n-1)\unitvector^{\top}z_-}\Big\}
\end{align*}
where $z=z_+-z_-$ is the unique decomposition of $z$ into positive and negative part, i.e.~$z_+=(\max\{z_j,0\})_{j=1}^n,z_-=(-z)_+=(\max\{-z_j,0\})_{j=1}^n)\in\mathbb R_+^n$.% satisfy $z_+^\top z_-=0$. %Also $\Delta^{n-1}$ is the standard simplex in $n$ dimensions and $p_{[l]}$ are the components of $p$ in decreasing order (cf.~introduction).
\end{lem}
\begin{proof}
For what follows let $ \hat z:=(\unitvector^{\top}z_+,-\unitvector^{\top}z_-,0,\ldots,0)$.

`` $\supseteq$ '': Majorization by definition forces $\unitvector^{\top}x=\unitvector^{\top}z_+-\unitvector^{\top}z_-=\unitvector^{\top}(z_+-z_-)=\unitvector^{\top}z$. Also if $x\prec \hat z$ then there exists a doubly stochastic matrix $A$ which maps $\hat z$ to $x$ so using \eqref{eq:doubly_stoch_trace_norm} we compute
$$
\|x\|_1=\|A\hat z\|_1\leq \|\hat z\|_1= \unitvector^{\top}z_++\unitvector^{\top}z_-=\sum\nolimits_{j=1}^n |z_j|=\|z\|_1\,.
$$

`` $\subseteq$ '': Decompose $x=x_+-x_-$ with $x_+,x_-\in\mathbb R_+^n$ as above. By assumption
\begin{align*}
\unitvector^{\top}x=\unitvector^{\top}x_+-\unitvector^{\top}x_-&=\unitvector^{\top}z_+-\unitvector^{\top}z_-=\unitvector^{\top}z\\
\|x\|_1=\unitvector^{\top}x_++\unitvector^{\top}x_-&\leq \unitvector^{\top}z_++\unitvector^{\top}z_-=\|z\|_1
\end{align*}
so taking the sum of these conditions gives $\unitvector^{\top}x_+\leq \unitvector^{\top}z_+$. Thus for all $k=1,\ldots,n-1$
$$
\sum\nolimits_{i=1}^k x_{[i]} \leq \sum\nolimits_{i=1}^k (x_{[i]})_+ \leq \unitvector^{\top}x_+\leq \unitvector^{\top}z_+= \unitvector^{\top}z_++\underbrace{0+\ldots+0}_{k-1\text{ zeros}} =\sum\nolimits_{i=1}^k \hat z_{[i]} 
$$
which---together with $\unitvector^{\top}x=\unitvector^{\top}z$---shows $x\prec\hat z$. 
%This by Prop.~\ref{lemma_char_d_vec} (iv) for $d=\unitvector$ readily shows $(\unitvector^{\top}x_+,-\unitvector^{\top}x_-,0,\ldots,0)^\top \prec (\unitvector^{\top}z_+,-\unitvector^{\top}z_-,0,\ldots,0)^\top $. On the other hand, again by Prop.~\ref{lemma_char_d_vec} (iv), every $x\in\mathbb R^n$ satisfies $x\prec (\unitvector^{\top}x_+,-\unitvector^{\top}x_-,0,\ldots,0)^\top $. Because $s_d(n)$ for all $d\in\mathbb R_{++}^n$ (so in particular for $d=\unitvector$) forms a semigroup this shows
%$$
%x\prec{\footnotesize\begin{pmatrix} \unitvector^{\top}x_+\\-\unitvector^{\top}x_-\\0\\\vdots\\0 \end{pmatrix}}\prec{\footnotesize\begin{pmatrix} \unitvector^{\top}z_+\\-\unitvector^{\top}z_-\\0\\\vdots\\0 \end{pmatrix}}\,.\eqno\tag*{\qedhere}
%$$
\end{proof}
The previous lemma is the key to transferring the $\mathscr H$-description of classical majorization over to $d$-majorization:
\begin{thm}\label{thm_maj_halfspace}
Let $y\in\mathbb R^n$, $d\in\mathbb R_{++}^n$. Then
$
M_d(y)=\{x\in\mathbb R^n : Mx\leq b\}
$
with $M$ being the matrix \eqref{eq:M_maj} and
\begin{equation}\label{eq:d_majorization_b}
b= \min_{i=1,\ldots,n} {\footnotesize\begin{pmatrix} 
\unitvector^{\top}(y-\frac{y_i}{d_i}d)_+\unitvector+\frac{y_i}{d_i}M_1d\\
\vdots\\
 \unitvector^{\top}(y-\frac{y_i}{d_i}d)_+\unitvector+\frac{y_i}{d_i}M_{n-1}d\\
\unitvector^{\top}y\\
-\unitvector^{\top}y
\end{pmatrix}} \in\mathbb R^{2^n}%\quad\text{ (cf.~footnote \ref{footnote_min_convention})}
\,.
\end{equation}
\end{thm}
\begin{proof}
Using \cite[Thm.~1]{Dahl10} as well as \eqref{eq:decomp_m_d}, Lemma \ref{lemma_trace_norm_ball_maj}, and Remark \ref{rem:halfspace_actions} we find
\begin{align*}
M_d(y)&=\bigcap\nolimits_{i=1}^n\Big(\Big\lbrace \tilde x\in\mathbb R^n : \unitvector^{\top}\tilde x=\unitvector^{\top}\Big(y-\frac{y_i}{d_i}d\Big)\ \wedge\ \|\tilde x\|_1\leq\Big\|y-\frac{y_i}{d_i}d\Big\|_1\Big\rbrace+\frac{y_i}{d_i}d\Big)\\
&=\bigcap\nolimits_{i=1}^n\Big(\Big\lbrace x\in\mathbb R^n : x\prec{\footnotesize\begin{pmatrix} \unitvector^{\top}(y-\frac{y_i}{d_i}d)_+\\-\unitvector^{\top}(y-\frac{y_i}{d_i}d)_-\\0\\\vdots\\0 \end{pmatrix}}\Big\}+\frac{y_i}{d_i}d\Big)\\
&=\bigcap\nolimits_{i=1}^n\Big\lbrace x\in\mathbb R^n : Mx\leq {\footnotesize\begin{pmatrix} \unitvector^{\top}(y-\frac{y_i}{d_i}d)_+\\\vdots\\ \unitvector^{\top}(y-\frac{y_i}{d_i}d)_+\\\unitvector^{\top}(y-\frac{y_i}{d_i}d)\\-\unitvector^{\top}(y-\frac{y_i}{d_i}d)\end{pmatrix}}+\frac{y_i}{d_i}Md \Big\}\\
&=\Big\lbrace x\in\mathbb R^n : Mx\leq \min_{i=1,\ldots,n}{\footnotesize\begin{pmatrix} 
\unitvector^{\top}(y-\frac{y_i}{d_i}d)_+\unitvector+\frac{y_i}{d_i}M_1d\\
\vdots\\
 \unitvector^{\top}(y-\frac{y_i}{d_i}d)_+\unitvector+\frac{y_i}{d_i}M_{n-1}d\\
\unitvector^{\top}y\\
-\unitvector^{\top}y
\end{pmatrix}} \Big\}\,.\tag*{\qedhere}
\end{align*}
\end{proof}
%\begin{rmk}
%BAUSTELLE
%\begin{itemize}
%\item[(i)] 
%
%
%
%\marginpar{``last equation''}
\noindent Setting $d=\unitvector$ in Thm.~\ref{thm_maj_halfspace}---together with Lemma \ref{lemma_maj_sum_recovery} (iii) (\ref{app_proofs_B})---recovers the known $\mathscr H$-description of classical majorization \cite[Thm.~1]{Dahl10} as expected.\smallskip
%
%
%
%\item[(ii)] While every $d$-majorization polytope is precisely the set of solutions to $Mx\leq b$ for some $b$ the converse is not true, i.e.~there exist $b\in\mathbb R^n$ such that $Mx\leq b$ has solutions but $\{x\in\mathbb R^n : Mx\leq b\}\neq M_d(y)$ for any $d\in\mathbb R_{++}^n$, $y\in\mathbb R^n$. For this let $n=2$ and $b=(-2,2,-1,1)^\top $. Then
%$$
%\{x\in\mathbb R^n : Mx\leq b\}=\operatorname{conv}\Big\{{\footnotesize \begin{pmatrix} -3\\2 \end{pmatrix},\begin{pmatrix} -2\\1 \end{pmatrix} }\Big\}
%$$
%as can be easily seen. Because this set only comprises indefinite vectors but every $M_d(y)$ contains $\frac{\unitvector^{\top}y}{\unitvector^{\top}d}d$ ($\geq 0$ or $\leq 0$ because $d>0$, cf.~Thm.~\ref{prop_1} (i)) the above set cannot be described by any $d$-majorization polytope.\marginpar{does this hold if displacements are allowed?}
%\end{itemize}
%\end{rmk}

The previous theorem shows that, roughly speaking, $\prec_d$ and $\prec$ share the same geometry, i.e.~the faces of $M_d(y)$ for arbitrary $y\in\mathbb R^n$, $d\in\mathbb R_{++}^n$ are all parallel to some face of a classical majorization polytope, but the precise location of the halfspaces (respectively faces) may differ. 
%This suggests that some properties of the $\prec$-polytope transfer to the general case which shall be explored now.

\begin{rmk}
The description of $d$-majorization via halfspaces is not only conceptionally interesting, it also enables an algorithmic computation of the extreme points of $M_d(y)$. In general, the problem of converting an $\mathscr H$-description to a $\mathscr V$-description is known as vertex enumeration problem and well-studied in the field of convex polytopes and computational geometry, see \cite{Avis97} for an overview. For arbitrary polytopes this is a hard problem but in our case---due to the particular structure of $M_d(y)$---one can achieve an explicit (even analytic) solution, cf.~Section \ref{sec_d_poly_analysis}.
%Example \ref{example_2} (\ref{app_b}).
%\item[(ii)] Note that Prop.~\ref{prop_maj_corner} is a mere consequence of the more well-behaving geometry of the $d$-majorization polytope in less than four dimensions. To see that this fails for $n\geq 4$ consider the counterexample given in Example \ref{example_2a} (\ref{app_b}).
%\end{itemize}
\end{rmk}

\subsection{Geometric and Topological Properties of the $d$-Majorization Polytope}\label{subsec_43}

If $M_d$ acts on a set consisting of more than one vector we can state further geometric and topological results. This will be of use when treating continuity questions of the map $(d,P)\mapsto M_d(P)$ afterwards.
\begin{thm}\label{coro_0_1}
Let $d\in\mathbb R_{++}^n$ and an arbitrary subset $P\subseteq \mathbb R^n$ be given. Then the following statements hold.
\begin{itemize}
\item[(i)] If $P$ lies within a trace hyperplane, i.e.~there exists $ c\in\mathbb R$ such that $\unitvector^{\top}x= c$ for all $x\in P$, then $M_d(P)$ is star-shaped with respect to $\frac{ c}{\unitvector^{\top}d}d$.
\item[(ii)] If $P$ is convex, then $M_d(P)$ is path-connected.
\item[(iii)] If $P$ is compact, then $ M_d(P)$ is compact.
\end{itemize}
\end{thm}
\begin{proof}
(i): Every $x\in P$ is directly connected to $\frac{\unitvector^{\top}x}{\unitvector^{\top}d}d=\frac{c}{\unitvector^{\top}d}d$ within $M_d(P)$ (Thm.~\ref{prop_1} \& convexity of $\prec_d$). (ii): Let $y,z\in P$, $\lambda\in[0,1]$ be arbitrary. Then $\lambda y+(1-\lambda)z\in P$ and
$$
\lambda \frac{\unitvector^{\top}y}{\unitvector^{\top}d}d+(1-\lambda)\frac{\unitvector^{\top}z}{\unitvector^{\top}d}d=\frac{\unitvector^{\top}(\lambda y+(1-\lambda)z)}{\unitvector^{\top}d}d\in M_d( \lambda y+(1-\lambda)z)\subseteq M_d(P)\,.
$$
%where in the last step we used that $M_d$ is increasing. Thus $\frac{\unitvector^{\top}y}{\unitvector^{\top}d}d,\frac{\unitvector^{\top}z}{\unitvector^{\top}d}d$ are path-connected in $M_d(P)$, which together with (i) shows (ii).
Thus $\frac{\unitvector^{\top}y}{\unitvector^{\top}d}d$ and $\frac{\unitvector^{\top}z}{\unitvector^{\top}d}d$ are
path-connected in $M_d(P)$, hence (ii) follows from (i).

(iii): As matrix multiplication is continuous, $M_d(P)=\{Az:A\in s_d(n),z\in P\}$ is compact as the image of the compact set $s_d(n)\times P$ under a continuous function.
%%As all norms on any finite-dimensional $\mathbb C$-linear space are equivalent (such as then all induced topologies) we can w.l.o.g. equip $\mathbb R^n$ with the 1-norm.
%As $P$ by assumption is bounded and $s_d(n)$ is bounded (because compact, cf.~Remark \ref{rem_maj_vector} (iii)) this readily implies that $M_d(P)$ is bounded. For closedness, consider a sequence $( x_n)_{n\in\mathbb N}$ in $ M_d(P)$ which converges to some $ x\in\mathbb R^n$. By definition there exists a sequence $(y_n)_{n\in\mathbb N}$ in $P$ and a sequence $(A_n)_{n\in\mathbb N}$ in $s_d(n)$ such that $A_ny_n=x_n$. Because $P$ is compact there exists a subsequence $(y_{m_j})_{j\in\mathbb N}$ of $(y_n)_{n\in\mathbb N}$ which converges to some $y\in P$. On the other hand compactness of $s_d(n)$ yields a subsequence $(A_{n_l})_{l\in\mathbb N}$ of $(A_{m_j})_{j\in\mathbb N}$ which converges to some $A\in s_d(n)$. Combining these two
%%---together with Remark \ref{rem_maj_vector}---
%yields
%%further subsequences $(X_{n_l})_{l\in\mathbb N},(T_{n_l})_{l\in\mathbb N}$ with coinciding index set which satisfy
%%\begin{align*}
%%\|Ay-A_{n_l}y_{n_l}\|_1&\leq \|Ay-A_{n_l}y\|_1+\|A_{n_l}y-A_{n_l}y_{n_l}\|_1\\
%%&\leq \|A-A_{n_l}\|\|y\|_1+\|y-y_{n_l}\|_1\to 0\quad\text{ as }l\to\infty\,.
%%\end{align*}
%%Therefore 
%$
% x=\lim_{l\to\infty} x_{n_l}=\lim_{l\to\infty}A_{n_l}y_{n_l}=Ay\,,
%$
%so $ x\in M_d(P)$ because $y\in P$ which concludes the proof. 
\end{proof}
\noindent The previous theorem still holds when extending $\prec_d$ to complex 
vectors. 
Also one might hope that Thm.~\ref{coro_0_1} (ii) is not optimal in the sense that 
convexity of general $P$ implies convexity of $M_d(P)$. Example \ref{example_1} 
(\ref{app_b}), however, gives a negative answer.\smallskip

Now the description of $M_d(y)$ as a convex polytope is powerful enough to answer continuity questions regarding the map $M_d$.
\begin{thm}\label{thm_cont_Md}
Let $\mathcal P_c(\mathbb R^n)$ denote the collection of all compact subsets of $\mathbb R^n$ and let $\delta$ be the Hausdorff metric\footnote{
Given a metric space $(X,d)$ and $A,B\subseteq X$ non-empty and compact, the Hausdorff distance
$$
\delta(A,B):=\max\big\{\max_{z\in A}\min_{w\in B} d(z,w),\max_{z\in B}\min_{w\in A} d(z,w)\big\}
$$
is a metric on the space of all non-empty compact subsets of $X$, cf.~\cite[§21.VII]{Kuratowski66}.
%Also convergence in this metric interchanges with taking the convex hull: if a bounded sequence $(A_n)_{n\in\mathbb N}$ of non-empty compact subsets of $X=\mathbb R^n$ converges to $A$ w.r.t.~the Hausdorff metric then $(\operatorname{conv}(A_n))_{n\in\mathbb N}$ converges to $\operatorname{conv}(A)$, cf.~\cite[Lemma 2.5 (b)]{DvE18}. (Note that the proof given there is for $X=\mathbb C$ but holds analogously for $X=\mathbb R^n$ because it is a simple consequence of Carathéodory's theorem.)
\label{footnote_hausdorff}
}
on $\mathcal P_c(\mathbb R^n)$ with respect to $\|\cdot\|_1$. Then the following statements hold.
\begin{itemize}
\item[(i)] For all $d\in\mathbb R_{++}^n$, $M_d$ is non-expansive under $\delta$, that is, for all $P,P'\in\mathcal P_c(\mathbb R^n)$ one has $\delta(M_d(P),M_d(P'))\leq\delta(P,P')$.
\item[(ii)] The following map is continuous:
\begin{align*}
M:\mathbb R_{++}^n\times(\mathcal P_c(\mathbb R^n),\delta)&\to (\mathcal P_c(\mathbb R^n),\delta)\\
(d,P)&\mapsto M_d(P)
\end{align*}
\end{itemize}
\end{thm}
\begin{proof}
 Note that the image of a compact set under $M_d$ remains compact by Thm.~\ref{coro_0_1} (iii) so this
 guarantees that the image of the map $M$ is contained in $\mathcal P_c(\mathbb R^n)$ and that $\delta(M_d(P),M_d(P'))$
 is well-defined. (i): As a direct consequence of \eqref{eq:doubly_stoch_trace_norm} one has
\begin{align*}
\max_{z\in M_d(P)}\min_{w\in M_d(P')}\|z-w\|_1&=\max_{\substack{A\in s_d(n)\\z_1\in P}}\min_{\substack{B\in s_d(n)\\z_2\in P'}}\|Az_1-Bz_2\|_1
\leq \max_{\substack{A\in s_d(n)\\z_1\in P}}\min_{z_2\in P'}\|A(z_1-z_2)\|_1\\
&\leq \max_{\substack{A\in s_d(n)\\z_1\in P}}\min_{z_2\in P'}\|z_1-z_2\|_1=\max_{z\in P}\min_{w\in P'}\|z-w\|_1\,.
\end{align*}
Interchanging $P$ and $P'$ yields the desired estimate.\smallskip

(ii): Our proof can be divided into the following five steps.\smallskip

\noindent\textit{Step 1:} For all $y\in\mathbb R^n$ the vector $b$ from \eqref{eq:d_majorization_b} continuously depends on $d\in\mathbb R_{++}^n$
%.
%
%This is due to the following facts:
%\begin{itemize}
%\item The map $x\mapsto x_+=\frac{x+|x|}{2}$ is continuous on $\mathbb R$.
%\item The map $x\mapsto \frac{1}{x}$ is continuous on $\mathbb R_{++}$.
%\item The map $(x_1,\ldots,x_n)\mapsto x_j$ for all $j=1,\ldots,n$ is continuous on $\mathbb R^n$.
%\end{itemize}
%Thus every component $b'_i:\mathbb R_{++}^n\to\mathbb R$, $d\mapsto b'_i(d)$ (the latter still corresponding to the vector $b$ from \eqref{eq:d_majorization_b}) is continuous
as a composition and a finite sum of continuous functions, using that the minimum over finitely many continuous functions remains continuous. Here we use that $d>0$ so $d\mapsto \frac1d$ is continuous.\smallskip

\noindent\textit{Step 2:} If a sequence $(b^{(m)})_{m\in\mathbb N}\subset\mathbb R^{2^n}$ with $b^{(m)}_{2^n-1}+b^{(m)}_{2^n}=0$ for all $m\in\mathbb N$ converges to $b\in\mathbb R^{2^n}$ in norm and all the induced convex polytopes $\{x\in\mathbb R^n : Mx\leq b^{(m)}\}$, $\{x\in\mathbb R^n : Mx\leq b\}$ are non-empty, then
$
\lim_{m\to\infty}\delta( \{x\in\mathbb R^n : Mx\leq b^{(m)}\},\{x\in\mathbb R^n : Mx\leq b\} )=0
$.

%All the $\{x\in\mathbb R^n : Mx\leq b^{(m)}\}$ are convex polytopes by Lemma \ref{lemma_Mx_b_convex_polytope} so in particular they are compact and we are allowed to consider their Hausdorff distance. 

This follows directly from \cite[Thm.~2.4]{LiWu93}---or, originally, %(\ref{app_sec_subsec43}) 
\cite{Hoffman52}---which yields a constant $c_M>0$ (only depending on $M$) such that
$$
\delta( \{x\in\mathbb R^n : Mx\leq b^{(m)}\},\{x\in\mathbb R^n : Mx\leq b\} )\leq c_M\|b^{(m)}-b\|_1\overset{m\to\infty}\to 0\,.
$$

\noindent \textit{Step 3:} $d\mapsto M_d(y)$ is continuous on $\mathbb R_{++}^n$ for all $y\in\mathbb R^n$.

Let $d^{(m)}\subset\mathbb R_{++}^n$ be a sequence with limit $d\in\mathbb R_{++}^n$. As shown in Step 1 this implies that $b^{(m)}=b(d^{(m)})\subset\mathbb R_{++}^{2^n}$ converges to $b(d)\in\mathbb R_{++}^{2^n}$ so Step 2 together with Thm.~\ref{thm_maj_halfspace} yields
 $$
\lim_{m\to\infty} M_{d^{(m)}}(y)=\lim_{m\to\infty} \{x\in\mathbb R^n : Mx\leq b^{(m)}\}=\{x\in\mathbb R^n : Mx\leq b\}=M_d(y)\,.
$$
%This proves continuity because the subspace topology (inherited from $(\mathbb R^n,\|\cdot\|_1$) on $\mathbb R_{++}^n$ is induced by the restricted norm $\|\cdot\|_1:\mathbb R_{++}^n\to \mathbb R_+$; thus $\mathbb R_{++}^n$ is a normed (hence metric) space and continuity is the same as sequential continuity.
\noindent \textit{Step 4:} $d\mapsto M_d(P)$ is continuous on $\mathbb R_{++}^n$ for all $P\in\mathcal P_c(\mathbb R^n)$.

As before let $d^{(m)}\subset\mathbb R_{++}^n$ be a sequence which converges to $d\in\mathbb R_{++}^n$ and let $\varepsilon>0$ be given. Because $P$ is compact one finds $y_1,\ldots,y_k\in P$, $k\in\mathbb N$ with $P\subseteq\bigcup_{i=1}^k B_{\varepsilon/2}(y_i)$. On the other hand (by Step 2) for every $i=1,\ldots,k$ one finds $N_i\in\mathbb N$ such that $ \delta(M_{d^{(m)}}(y_i),M_d(y_i))<\frac{\varepsilon}{2} $ for all $m\geq N_i$. We want to show $\delta(M_{d^{(m)}}(P),M_d(P))<\varepsilon$ for all $m\geq N:=\max\{N_1,\ldots,N_k\}$ which would imply the claim. 

Let $m\geq N$ and $x\in M_{d^{(m)}}(P)$ so one finds $A\in s_{d^{(m)}}(n)$ and 
$y\in P$ such that $x=Ay$. First compactness of $P$ yields $y_i \in P$ such that
$\|y-y_i\|<\frac{\varepsilon}{2}$. Then $Ay_i$ is in $ M_{d^{(m)}}(y_i)$ which lets us 
pick $\tilde x\in M_d(y_i) \subset M_d(P)$ with $\|Ay_i- \tilde x\|_1<\frac{\varepsilon}{2}$ (because 
$\delta(M_{d^{(m)}}(y_i),M_d(y_i))<\frac{\varepsilon}{2}$). Using 
\eqref{eq:doubly_stoch_trace_norm} we compute
\begin{align*}
\|x-\tilde x\|_1\leq \|Ay-Ay_i\|_1+\|Ay_i-\tilde x\|_1\leq \|y-y_i\|_1+\|Ay_i-\tilde x\|_1<\frac{\varepsilon}{2}+\frac{\varepsilon}{2}=\varepsilon\,.
\end{align*}
Analogously for every $\tilde x\in M_{d}(P)$ one finds $x\in M_{d^{(m)}}(P)$ such that $\|x-\tilde x\|_1<\varepsilon$ which by definition of $\delta$ implies $\delta(M_{d^{(m)}}(P),M_d(P))<\varepsilon$ for all $m\geq N$.\smallskip

\noindent \textit{Step 5:} $M$ is continuous (in the product topology).

%As both spaces which make up the domain of $M$ are metric spaces the product topology is metrizable and continuity, again, can be decided via sequences.
Let $(d^{(m)},P_m)_{m\in\mathbb N}\subset \mathbb R_{++}^n\times \mathcal P_c(\mathbb R^n)$ converge to $(d,P)\in \mathbb R_{++}^n\times \mathcal P_c(\mathbb R^n)$ in the product topology, i.e.~$\lim_{m\to\infty}\|d^{(m)}-d\|_1=0$ and $\lim_{m\to\infty}\delta(P^{(m)},P)=0$. By Step 4 the former implies $\lim_{m\to\infty}\delta(M_{d^{(m)}}(P),M_d(P))=0$ so using
 that $M_{d^{(m)}}$ is non-expansive we find
\begin{align*}
\delta(M_{d^{(m)}}(P^{(m)}),M_d(P))&\leq \delta(M_{d^{(m)}}(P^{(m)}),M_{d^{(m)}}(P))+\delta(M_{d^{(m)}}(P),M_d(P))\\
&\leq \delta(P^{(m)},P)+\delta(M_{d^{(m)}}(P),M_d(P))\overset{m\to\infty}\to 0\,.\tag*{\qedhere}
\end{align*}
\end{proof}
\begin{rmk}
\begin{itemize}
\item[(i)]

%ARXIV ONLY
Continuity of the map $M$ is supported by the fact that the half-spaces limiting $M_d(y)$ are independent of $d,y$. An example of a discontinuous relation between $A\in\mathbb R^{m\times n}$ and the induced polytope $\{x\in\mathbb R^n : Ax\leq b\}$ can be found in Example \ref{example_not_cont_A}.

\item[(ii)]
While $M_d$ is, in principle, defined for arbitrary $d\in\mathbb R^n$ the continuity statement from Thm.~\ref{theorem_max_corner_maj} (ii) fails if the domain is extended to $\mathbb R_+^n\times P_c(\mathbb R^n)$. A counterexample is given in Example \ref{ex_discont_IR_plus} (\ref{app_b}).
\end{itemize}
\end{rmk}

\subsection{Analyzing the $d$-Majorization Polytope}\label{sec_d_poly_analysis}

So far we learned that the majorization polytope $M_d(y)$ induced by a single vector $y\in\mathbb R^n$ with respect to some $d\in\mathbb R_{++}^n$ differs from the classical majorization polytope not in the orientation of the faces but only in their precise location. By Thm.~\ref{thm_maj_halfspace} this difference is fully captured by the following map:
\begin{equation}\label{eq:f_b'_vec}
\begin{split}
f:[0,\unitvector^{\top}d]&\to\mathbb R\\
c&\mapsto \min_{i=1,\ldots,n} \Big(\unitvector^{\top}\Big(y-\frac{y_i}{d_i}d\Big)_++\frac{y_i}{d_i}c\Big)
\end{split}
\end{equation}
Thus if we want to learn more about the $d$-majorization polytope we are well-advised to study the properties of \eqref{eq:f_b'_vec}. 
\begin{lem}\label{lemma_properties_of_f_min}
Let $y\in\mathbb R^n$, $d\in\mathbb R_{++}^n$ be given and let $\sigma\in S_n$ be any permutation which orders $(\frac{y_i}{d_i})_{i=1}^n$ decreasingly, i.e.~$\frac{y_{\sigma(1)}}{d_{\sigma(1)}}\geq\ldots\geq\frac{y_{\sigma(n)}}{d_{\sigma(n)}}$. Then the map \eqref{eq:f_b'_vec} has the following properties.
\begin{itemize}
\item[(i)] $f$ is continuous, piecewise linear, and concave.
\item[(ii)] For arbitrary $j=1,\ldots,n$ and $c\in(\sum_{i=1}^{j-1}d_{\sigma(i)},\sum_{i=1}^{j}d_{\sigma(i)})$
\begin{equation*}%\label{eq:action_f_min}
f(c)=\sum\nolimits_{i=1}^{j-1} y_{\sigma(i)}+\frac{y_{\sigma(j)}}{d_{\sigma(j)}}\Big(c-\sum\nolimits_{i=1}^{j-1}d_{\sigma(i)}\Big)
\end{equation*}
as well as $f'(c)=\frac{y_{\sigma(j)}}{d_{\sigma(j)}}$ so the (weak) derivative of $f$ is monotonically decreasing.
\item[(iii)] For all $j=0,\ldots,n$ one has $f( \sum_{i=1}^j d_{\sigma(i)} )= \sum_{i=1}^j y_{\sigma(i)} $ so in particular $f(0)=0$ and $f(\unitvector^{\top}d)=\unitvector^{\top}y$.
\item[(iv)] Let $k=1,\ldots,n-1$, $\tau\in S_n$, and pairwise different $\alpha_1,\ldots,\alpha_k\in\{1,\ldots,n\}$ be given. Then
$$
\sum\nolimits_{j=1}^k f\Big(\sum\nolimits_{i=1}^{\alpha_j-1}d_{\tau(i)}\Big)+f\Big(\sum\nolimits_{i=1}^k d_{\tau(\alpha_i)}\Big)\geq\sum\nolimits_{j=1}^k f\Big(\sum\nolimits_{i=1}^{\alpha_j}d_{\tau(i)}\Big)\,.
$$
%Let $k\in\mathbb N$ and $x\in(\unitvector^{\top}d)\Delta^{k-1}$. Then 
%$
%\unitvector^{\top}y\leq\sum\nolimits_{j=1}^nf(x_j)\,.
%$
\item[(v)] For all $j=1,\ldots,n-1$
$$
\frac{f\big(\sum\nolimits_{i=1}^j d_{[i]}\big)-f\big(\sum\nolimits_{i=1}^{j-1} d_{[i]}\big)}{d_{[j]}}\geq \frac{f\big(\sum\nolimits_{i=1}^{j+1} d_{[i]}\big)-f\big(\sum\nolimits_{i=1}^j d_{[i]}\big)}{d_{[j+1]}}\,.
$$
%\item[(vi)] Let $y\in\mathbb R_+^n$ and $\tau\in S_n$ be a permutation such that $d_{\tau(1)}\geq\ldots\geq d_{\tau(n)}$ as well as $f(\sum_{i=1}^j d_{\tau(i)})=\sum_{i=1}^j y_{\tau(i)}$ for all $j=1,\ldots,n-1$. Then $\frac{y_{\tau(1)}}{d_{\tau(1)}}\geq\ldots\geq \frac{y_{\tau(n)}}{d_{\tau(n)}}$.
\end{itemize}
\end{lem}
\begin{proof}
(i): The minimum over finitely many affine linear functions (in particular these functions are continuous \& concave) is piecewise linear, continuous, and concave. (ii): Direct consequence of Lemma \ref{lemma_maj_sum_recovery} (\ref{app_proofs_B}). (iii): Follows from (ii) together with continuity of $f$. (iv): Because $f$ is continuous \& concave ($-f$ is continuous \& convex) this is a direct consequence of Prop.~\ref{lemma_char_d_vec} (ii) (for $d=\unitvector$) together with Lemma \ref{lemma_d_ordering_maj} (\ref{app_proofs_B}) and $f(0)=0$ from (iii). (v): Define $\tilde d:=(d_{[j]},d_{[j+1]})^\top \in\mathbb R_{++}^2$. Evidently
$$
\Big(\sum\nolimits_{i=1}^j d_{[i]}\Big)\tilde d=\begin{pmatrix}
 d_{[j]}\sum\nolimits_{i=1}^j d_{[i]}\\
 d_{[j+1]}\sum\nolimits_{i=1}^j d_{[i]}
\end{pmatrix}\prec_{\tilde d} \begin{pmatrix} d_{[j]}\sum\nolimits_{i=1}^{j+1} d_{[i]} \\
d_{[j+1]}\sum\nolimits_{i=1}^{j-1} d_{[i]}%+(d_{[j]}-d_{[j+1]})
\end{pmatrix}
$$
due to minimality of $\tilde d$ w.r.t.~$\prec_{\tilde d}$ (Thm.~\ref{prop_1} (i)) and because the entries of the two vectors sum up to the same.
%\footnote{Direct computation:
%\begin{align*}
%d_{[j]}\sum\nolimits_{i=1}^j d_{[i]}+ d_{[j+1]}\sum\nolimits_{i=1}^j d_{[i]}=d_{[j]}\sum\nolimits_{i=1}^j d_{[i]}+ d_{[j+1]} d_{[j]}+ d_{[j+1]}\sum\nolimits_{i=1}^{j-1} d_{[i]}= d_{[j]}\sum\nolimits_{i=1}^{j+1}d_{[i]}+ d_{[j+1]}\sum\nolimits_{i=1}^{j-1} d_{[i]}
%\end{align*}
%}
Again Prop.~\ref{lemma_char_d_vec} (ii) for $d\to\tilde d$ yields
\begin{align*}
d_{[j]} f\Big(\sum\nolimits_{i=1}^j d_{[i]}\Big)+d_{[j+1]} f\Big(\sum\nolimits_{i=1}^j d_{[i]}\Big)&=
d_{[j]} f\Big(\frac{ d_{[j]}\sum\nolimits_{i=1}^j d_{[i]}}{d_{[j]}}\Big)+d_{[j+1]} f\Big( \frac{d_{[j+1]}\sum\nolimits_{i=1}^j d_{[i]}}{d_{[j+1]}}\Big)\\
&\geq d_{[j]} f\Big(\frac{ d_{[j]}\sum\nolimits_{i=1}^{j+1} d_{[i]}}{d_{[j]}}\Big)+d_{[j+1]} f\Big( \frac{d_{[j+1]}\sum\nolimits_{i=1}^{j-1} d_{[i]}}{d_{[j+1]}}\Big)\\
&= d_{[j]} f\Big(\sum\nolimits_{i=1}^{j+1} d_{[i]}\Big)+d_{[j+1]} f\Big(\sum\nolimits_{i=1}^{j-1} d_{[i]}\Big)
\end{align*}
because $-f$ is convex, which readily implies (v).
\end{proof}

For all rows $m\in\{0,1\}^n$ of $M$, the $b$-vector of the $d$-majorization polytope satisfies $\mathfrak b(m)=f(md)$ so, in slight abuse of notation, $M_d(y)=\{x\in\mathbb R^n : Mx\leq f(Md)\}$. 
\begin{rmk}
Recall that in the physics literature, thermo-majorization is usually defined via curves of the following form: Given any vector $z\in\mathbb R^n$ and $d\in\mathbb R_{++}^n$ consider the piecewise linear, continuous curve fully characterized by the elbow points $\{ \big(\sum_{i=1}^j d_{\sigma(i)},\sum_{i=1}^j z_{\sigma(j)}\big) \}_{j=0}^n$, where $\sigma$ is any permutation such that $\frac{z_{\sigma(1)}}{d_{\sigma(1)}}\geq\ldots\geq\frac{z_{\sigma(n)}}{d_{\sigma(n)}}$. Then a vector $y$ is said to thermo-majorize $x$ if $\unitvector^{\top} x=\unitvector^{\top} y$ and if the curve induced by $y$ is never below the curve induced by $x$ \cite{Horodecki13}.
But by the previous lemma this thermo-majorization curve is precisely the function $f$ which characterizes the polytope meaning $x\prec_dy$ is equivalent to $f_x(c)\leq f_y(c)$ for all $c\in[0,\unitvector^{\top}d]$; more on this in a bit.
\end{rmk}
While this confirms the (well-known) equivalence of $d$-majorization and thermo-majorization, we can reduce the comparison of the two curves to just the ``elbow points'' of the lower curve---as already observed in \cite[Thm.~4]{Alhambra16}---by means of the following elegant proof:
\begin{proof}[Proof of Prop.~\ref{lemma_char_d_vec} (i) $\Leftrightarrow$ (vii)]
By Lemma \ref{lemma_closure_op}, $x\prec_dy$ is equivalent to $M_d(x)\subseteq M_d(y)$ which by Thm.~\ref{thm_maj_halfspace} and Remark \ref{rem:halfspace_actions} holds if and only if $\unitvector^{\top}x=\unitvector^{\top}y$ and $f_x((Md)_i)\leq f_y((Md)_i)$ for all $i=1,\ldots,2^n-2$. Now we may apply Lemma \ref{lemma_properties_of_f_min} and the ``elbow point principle'' (i.e.~only check the elbow points of the \textit{lower concave} curve, similar to the proof of Prop.~\ref{lemma_char_d_vec}) to arrive at the equivalent condition: $\unitvector^{\top}x=\unitvector^{\top}y$ and $\sum_{i=1}^j x_{\sigma(i)}=f_x(\sum_{i=1}^jd_{\sigma(i)})\leq f_y(\sum_{i=1}^jd_{\sigma(i)})$ for all $j=1,\ldots,n-1$, which concludes the proof.
\end{proof}

Another advantage of introducing and studying the function $f$ is that its properties transfer to $M_d(y)$ which suffices to fully characterize the extreme points of the $d$-majorization polytope, thus generalizing \cite{Rado52}:

\begin{thm}\label{thm_Eb_sigma}
Let $y\in\mathbb R^n$, $d\in\mathbb R_{++}^n$. Then the extreme points of $M_d(y)$ are precisely the $E_b(\sigma)$, $\sigma\in S_n$ from Definition \ref{def_Ebsigma}. In particular $M_d(y)=\operatorname{conv}\{E_b(\sigma) : \sigma\in S_n\}$.
\end{thm}
\begin{proof}
If we can show $\{E_b(\sigma) : \sigma\in S_n\}\subseteq M_d(y)$ then by Thm.~\ref{thm_Ebsigma_extreme} we are done. Let arbitrary $\sigma\in S_n$ be given. Showing $E_b(\sigma)\in M_d(y)$ by Thm.~\ref{thm_maj_halfspace} is equivalent to showing $ME_b(\sigma)\leq b$, i.e.
$$
\sum_{i=1}^j (E_b(\sigma))_{\sigma(\alpha_i)}=\Big(\sum_{i=1}^j e_{\sigma(\alpha_i)}\Big)^\top E_b(\sigma)\leq \mathfrak b\Big(\sum_{i=1}^j e_{\sigma(\alpha_i)}^\top \Big)=f\Big(\sum_{i=1}^j d_{\sigma(\alpha_i)}\Big)
$$
for all $j=1,\ldots,n-1$ and all pairwise different $\alpha_1,\ldots,\alpha_j\in\{1,\ldots,n\}$ where $f$ is the map from \eqref{eq:f_b'_vec}. Be aware that writing $\sigma(\alpha_i)$ instead of $\alpha_i$ in above inequality
yields an equivalent problem as the $\alpha_i$ can be chosen arbitrarily anyway; this will be advantageous because the expression $(E_b(\sigma))_{\sigma(\alpha_i)}$ is easier to handle than $(E_b(\sigma))_{\alpha_i}$. 
Indeed by Lemma \ref{lemma_E_b_sigma_properties}
\begin{align*}
(E_b(\sigma))_{\sigma(\alpha_j)}
%(b'_\sigma)_{\alpha_j}-(b'_\sigma)_{\alpha_j-1}
=\mathfrak b\Big(\sum_{i=1}^{\alpha_j} e_{\sigma(i)}^\top \Big)-\mathfrak b\Big(\sum_{i=1}^{\alpha_j-1} e_{\sigma(i)}^\top \Big)
=f\Big(\sum_{i=1}^{\alpha_j} d_{\sigma(i)}\Big)-f\Big(\sum_{i=1}^{\alpha_j-1} d_{\sigma(i)}\Big)
\end{align*}
for all $j=1,\ldots,k$. Hence showing $E_b(\sigma)\in M_d(y)$ is equivalent to
\begin{align*}
\sum\nolimits_{j=1}^k \Big(f\Big(\sum\nolimits_{i=1}^{\alpha_j} d_{\sigma(i)}\Big)&-f\Big(\sum\nolimits_{i=1}^{\alpha_j-1} d_{\sigma(i)}\Big)\Big)\leq f\Big(\sum\nolimits_{i=1}^k d_{\sigma(\alpha_i)}\Big)%\\
%\Leftrightarrow\quad \sum\nolimits_{j=1}^k f\Big(\sum\nolimits_{i=1}^{\alpha_j} d_{\sigma(i)}\Big)&\leq \sum\nolimits_{j=1}^kf\Big(\sum\nolimits_{i=1}^{\alpha_j-1} d_{\sigma(i)}\Big)+f\Big(\sum\nolimits_{i=1}^k d_{\sigma(\alpha_i)}\Big)\
\end{align*}
which holds due to Lemma \ref{lemma_properties_of_f_min} (iv).
\end{proof}
%
%{\color{red}
%question: are the $E_b(\sigma)$ the only extreme points?
%\begin{conjecture}
%Let $y\in\mathbb R^n$, $d\in\mathbb R_{++}^n$. Then $M_d(y)=\operatorname{conv}\{E_b(\sigma) : \sigma\in S_n\}$. 
%\end{conjecture}
%
%This is the most natural generalization of Coro.~ \ref{coro_classical_maj_extreme} because $E_b(\sigma)$ for classical majorization is simply $\underline{\sigma}y$ (where $y\in\mathbb R^n$ is the initial vector). 
%Be aware that the above result, if it were true, cannot hold for general (non-empty) polytopes of the form $\{x\in\mathbb R^n : Mx\leq b\}$ where $M$ is the matrix from \eqref{eq:M_maj}. 
%
%most likely: property of $M$ (if $E_b(\sigma)\in P_M(b)$ $\forall_{\sigma\in S_n}$ for some $b\in\mathbb R^{2^n}$ with $b'_n +b'_{n+1}=0$ then there are no further extreme points)
%}

We immediately obtain the following result.

\begin{corollary}\label{thm_convex_poly}
Let $y\in\mathbb R^n$, $d\in\mathbb R_{++}^n$. Then $M_d(y)$ is a non-empty convex polytope of dimension at most $n-1$ and, moreover, has at most $n!$ extreme points. 
\end{corollary}
\begin{proof}
For non-emptiness note that $y\in M_d(y)$ because $\identity_n\in s_d(n)$. %The number of dimensions is implied by the trace condition $\unitvector^{\top}x=\unitvector^{\top}y$ for all $x\in M_d(y)$. 
By Thm.~\ref{thm_maj_halfspace} there exists $b\in\mathbb R^{2^n}$ such that $M_d(y)=\{x\in\mathbb R^n : Mx\leq b\}$ so $M_d(y)$ is a convex polytope of at most $n-1$ dimensions (Lemma \ref{lemma_Mx_b_convex_polytope}). Finally the extreme points of $M_d(y)$ are given by $\{E_b(\sigma) : \sigma\in S_n\}$ (Thm.~\ref{thm_Eb_sigma}) which due to $|S_n|=n!$ concludes the proof.
\end{proof}
\begin{rmk}
\begin{itemize}
\item[(i)] These results (Thm.~\ref{thm_Eb_sigma} \& Coro.~\ref{thm_convex_poly}) recently appeared in the physics literature for the special case $y\geq 0$ \cite[Sec.~2.2]{Alhambra19} but with an entirely different proof strategy: Alhambra et al.~explicitly constructed a family $\{P^{(\pi,\alpha)}\}_{\alpha}$ of $d$-stochastic matrices called ``$\beta$-permutations'' with the property that $\{P^{(\pi,\alpha)}y\}_{\alpha}$ contains all extreme points of $M_d(y)$, which---in our language---necessarily have to be of the form $E_b(\sigma)$ \cite[Lemma 12]{Lostaglio18}.
\item[(ii)] 
Analyzing and structuring the situations when $M_d(y)$ has less than $n!$ corners, i.e.~when $d,y$ are chosen such that the map $E_b$ is not injective reveals further connections between $M_d(y)$ and the map $f$ that determines the entries of the vector $b$. We will treat this question in a forthcoming paper \cite{PolytopeDegen22}.
\end{itemize}
\end{rmk}

Now one of these extreme points has the property of classically majorizing every other element inside the $d$-majorization polytope. The result, which is of particular interest, e.g., to tackle reachability questions in quantum control theory \cite{MTNS2020_1_alt,vE_PhD_2020}, reads as follows:

\begin{thm}\label{theorem_max_corner_maj}
Let $y\in\mathbb R_+^n$, $d\in\mathbb R_{++}^n$. Then there exists $z\in M_d(y)$ such that $x\prec z$ for all $x\in M_d(y)$, i.e.~$M_d(y)\subseteq M_\unitvector(z)$, and this $z$ is unique up to permutation.

More precisely if $\sigma\in S_n$ orders $d$ decreasingly, that is, $d_{\sigma(1)}\geq\ldots\geq d_{\sigma(n)}$, then one has $M_d(y)\subseteq M_\unitvector(E_b(\sigma))$. Thus $z$ can be chosen to be the extreme point $E_b(\sigma)$, i.e.~the solution to
\begin{equation}\label{eq:implicit_solution_z}
\Lambda\,\underline{\sigma}z=
\min_{i=1,\ldots,n}{\footnotesize\begin{pmatrix}
\unitvector^{\top}(y-\frac{y_i}{d_i}d)_++\frac{y_i}{d_i}d_{\sigma(1)}\\
\vdots\\
\unitvector^{\top}(y-\frac{y_i}{d_i}d)_++\frac{y_i}{d_i}\sum_{j=1}^{n-1}d_{\sigma(j)}\\
\unitvector^{\top}y
\end{pmatrix}} \overset{\eqref{eq:f_b'_vec}}=\Big(f\big(\sum\nolimits_{i=1}^jd_{\sigma(i)}\big)\Big)_{j=1}^n \,.
\end{equation}
Moreover $M_d(z)\subseteq M_\unitvector(z)$, and $\frac{z}{d}$ and $d$ are similarly ordered, i.e.~$\frac{z_{\sigma(1)}}{d_{\sigma(1)}}\geq \ldots\geq \frac{z_{\sigma(n)}}{d_{\sigma(n)}}$.%
%\begin{align*}
%z_{\sigma(1)}&=\min_{i=1,\ldots,n}\unitvector^{\top}(y-\frac{y_i}{d_i}d)_+\unitvector+\frac{y_i}{d_i}d_{\sigma(1)}\qquad\text{ as well as}\\
%z_{\sigma(k)}&=\min_{i=1,\ldots,n}\Big( \unitvector^{\top}(y-\frac{y_i}{d_i}d)_+\unitvector+\frac{y_i}{d_i}\sum_{j=1}^{k}d_{\sigma(j)} \Big)-\min_{i=1,\ldots,n}\Big( \unitvector^{\top}(y-\frac{y_i}{d_i}d)_+\unitvector+\frac{y_i}{d_i}\sum_{j=1}^{k-1}d_{\sigma(j)} \Big)
%\end{align*}
%for all $k=2,\ldots,n$. 
%\item[(iv)] If $y\in\mathbb R_{++}^n$ and if the entries of $d$ are pairwise distinct then $z$ is unique.
%\end{itemize}
\end{thm}
\begin{proof}
%W.l.o.g.~we may assume\footnote{
%One readily verifies the identity
%\begin{equation}\label{eq:perm_inv}
% M_d(y)=\underline{\tau}^{-1}M_{\underline{\tau} d}(\underline{\tau}y)
%\end{equation}
%for all %$y\in\mathbb R^n$, $d\in\mathbb R_{++}^n$, 
%$\tau\in S_n$. Now assume that the case of $\frac{y}{d}$ being ordered decreasingly is proven. For arbitrary $y,d$ there exists $\tau\in S_n$ with $(\frac{\underline{\tau}y}{\underline{\tau}d})_1=(\underline{\tau}\frac{y}{d})_1\geq\ldots\geq (\underline{\tau}\frac{y}{d})_n=(\frac{\underline{\tau}y}{\underline{\tau}d})_n$, thus one finds $z\in M_{\underline{\tau}d}(\underline{\tau}y)$ such that $M_{\underline{\tau}d}(\underline{\tau}y)\subseteq M_\unitvector(z)$. But then $\underline{\tau}^{-1}z$ is in $ M_d(y)$ by \eqref{eq:perm_inv}, solves \eqref{eq:implicit_solution_z} (on the l.h.s.~for $\sigma\to\tau^{-1}\circ\sigma$ because this is the permutation which orders $\underline{\tau}d$ decreasingly) and moreover
%$$
%M_d(y)=\underline{\tau}^{-1}M_{\underline{\tau} d}(\underline{\tau}y)\subseteq \underline{\tau}^{-1}M_\unitvector(z)=M_\unitvector(z)=M_\unitvector(\underline{\tau}^{-1}z)
%$$
%by permutation invariance of classical majorization.
%} $\frac{y_1}{d_1}\geq \frac{y_2}{d_2}\geq\ldots\geq \frac{y_n}{d_n}$. Although this is not necessary it should make the following computations a bit more clear.

Uniqueness of such $z$ (up to permutation) is the easiest to show: If there exist $z_1,z_2\in M_d(y)$ such that $M_d(y)\subseteq M_\unitvector(z_i)$ for $i=1,2$, then in particular $z_{2-i}\in M_d(y)\subseteq M_\unitvector(z_i)$. Hence $z_1\prec z_2\prec z_1$ and one finds a permutation $\tau\in S_n$ such that $z_1=\underline{\tau}z_2$.

For existence let $\sigma\in S_n$ be any permutation which orders $d$ decreasingly. From Thm.~\ref{thm_Eb_sigma} we know that $z:=E_b(\sigma)\in M_d(y)$.
By \eqref{eq:implicit_solution_z} %(or, equivalently, Lemma \ref{lemma_E_b_sigma_properties}) 
for all $j=1,\ldots,n-1$ one finds
\begin{equation}\label{eq:entry_z_sigma}
z_{\sigma(j)}=f\Big(\sum\nolimits_{i=1}^{j}d_{\sigma(i)}\Big)-f\Big(\sum\nolimits_{i=1}^{j-1}d_{\sigma(i)}\Big)=f\Big(\sum\nolimits_{i=1}^{j}d_{[i]}\Big)-f\Big(\sum\nolimits_{i=1}^{j-1}d_{[i]}\Big)\,.
\end{equation}
%where $f$ is the map from Lemma \ref{lemma_properties_of_f_min}. 
Therefore $\frac{z_{\sigma(j)}}{d_{\sigma(j)}}\geq \frac{z_{\sigma(j+1)}}{d_{\sigma(j+1)}}$ is equivalent to 
\begin{align*}
\frac{f\big(\sum\nolimits_{i=1}^j d_{[i]}\big)-f\big(\sum\nolimits_{i=1}^{j-1} d_{[i]}\big)}{d_{[j]}}\geq \frac{f\big(\sum\nolimits_{i=1}^{j+1} d_{[i]}\big)-f\big(\sum\nolimits_{i=1}^j d_{[i]}\big)}{d_{[j+1]}}
\end{align*}
which holds due to Lemma \ref{lemma_properties_of_f_min} (v), meaning $\frac{z}{d}$ and $d$ are indeed similarly ordered, as claimed. More importantly because $z\in\mathbb R_+^n$ (stochastic matrices preserve non-negativity of $y$) one even has $z_{\sigma(j)}=z_{[j]}$ for all $j=1,\ldots,n$ because
$
z_{\sigma(j)}\geq \frac{d_{\sigma(j)}}{d_{\sigma(j+1)}}z_{\sigma(j+1)}=\frac{d_{[j]}}{d_{[j+1]}} z_{\sigma(j+1)}\geq z_{\sigma(j+1)}
$.

Now recall that $M_\unitvector(z)=\{x\in\mathbb R^n : Mx\leq b'_z\}$ \cite[Thm.~1]{Dahl10} where $b'_z$ is of the following form:
the first $\binom{n}{1}$ entries equal $z_{[1]}=z_{\sigma(1)}$, the next $\binom{n}{2}$ entries equal $z_{[1]}+z_{[2]}=z_{\sigma(1)}+z_{\sigma(2)}$ and so forth until $\binom{n}{n-1}$ entries equaling $\sum_{i=1}^{n-1} z_{[i]}=\sum_{i=1}^{n-1} z_{\sigma(i)}$. Writing $M_d(y)=\{x\in\mathbb R^n : Mx\leq b\}$ (Thm.~\ref{thm_maj_halfspace}), if we can show that $b\leq b'_z$ then we get $M_d(y)\subseteq M_\unitvector(z)$ (Remark \ref{rem:halfspace_actions}) as desired.

For all $k=1,\ldots,n-1$ and all $\tau\in S_n$ by Lemma \ref{lemma_minmax} (\ref{app_proofs_B})---which we may apply because $\frac{y_i}{d_i}\geq 0$ for all $i$---we compute
\begin{align*}
\mathfrak b\Big(\sum_{j=1}^k e_{\tau(j)}^\top \Big)&=\min_{i=1,\ldots,n}\unitvector^{\top}\big(y-\frac{y_i}{d_i}d\big)_++\frac{y_i}{d_i}\sum_{j=1}^k d_{\tau(j)}\\
&\leq\max_{\tau\in S_n}\min_{i=1,\ldots,n}\unitvector^{\top}\big(y-\frac{y_i}{d_i}d\big)_++\frac{y_i}{d_i}\sum_{j=1}^k d_{\tau(j)}\\
&=\min_{i=1,\ldots,n}\unitvector^{\top}\big(y-\frac{y_i}{d_i}d\big)_++\frac{y_i}{d_i}\Big(\max_{\tau\in S_n}\sum_{j=1}^k d_{\tau(j)}\Big)\\
&=\min_{i=1,\ldots,n}\unitvector^{\top}\big(y-\frac{y_i}{d_i}d\big)_++\frac{y_i}{d_i}\Big(\sum_{j=1}^kd_{[j]}\Big)\overset{\eqref{eq:entry_z_sigma}}=\sum_{i=1}^k z_{\sigma(i)}=\mathfrak b'_z\Big(\sum_{j=1}^k e_{\tau(j)}^\top \Big)
\end{align*}
so $b\leq b'_z$ as claimed. To conclude the proof note that by Lemma \ref{lemma_closure_op} 
%
%(i): As argued above $z=E_b(\sigma)$ is an extreme point of $M_d(y)$ by Thm.~\ref{thm_Eb_sigma}.
$z\in M_d(y)$ implies $M_d(z)\subseteq (M_d\circ M_d)(y)=M_d(y)\subseteq M_\unitvector(z)$.
\end{proof}
\noindent Non-negativity of $y$ in Thm.~\ref{theorem_max_corner_maj} is actually necessary as Example \ref{ex_pos_y_necess} (\ref{app_b}) shows.\smallskip

Given our knowledge of this maximal point (w.r.t.~classical majorization) in the $d$-majorization polytope, one can now give a necessary condition for when the initial vector $y$ itself is this maximal element.

\begin{corollary}\label{coro_y_max_corner}
Let $y\in\mathbb R_+^n$, $d\in\mathbb R_{++}^n$. If $\frac{y}{d}$ and $d$ are similarly ordered, that is, there exists a permutation $\sigma\in S_n$ such that $d_{\sigma(1)}\geq \ldots\geq d_{\sigma(n)}$ and $\frac{y_{\sigma(1)}}{d_{\sigma(1)}}\geq \ldots\geq \frac{y_{\sigma(n)}}{d_{\sigma(n)}}$, then $M_d(y)\subseteq M_\unitvector(y)$.
%and $y_{\sigma(1)}\geq \ldots\geq y_{\sigma(n)}$
\end{corollary}
\begin{proof}
If $d$ and $\frac{y}{d}$ are similarly ordered (by means of $\sigma$) then, by Thm.~\ref{theorem_max_corner_maj}, $M_d(y)\subseteq M_\unitvector(E_b(\sigma))$ with
$
\Lambda\,\underline{\sigma}E_b(\sigma)=
(f(\sum_{i=1}^jd_{\sigma(i)}))_{j=1}^n=
(\sum_{i=1}^jy_{\sigma(i)})_{j=1}^n=
\Lambda\,\underline{\sigma}y
$
by Lemma \ref{lemma_properties_of_f_min} (iii), hence $E_b(\sigma)=y$.
\end{proof}
Be aware that the converse to Coro.~\ref{coro_y_max_corner} does not hold, refer to Example \ref{ex_counter_y_max}, \ref{app_b}.
%
%
%%ARXIV ONLY
To see how the $d$-majorization polytope behaves (aside from continuity) when changing only $d$ while leaving the initial vector $y$ untouched we refer to Example \ref{ex_wandering_d_vector}. This example also illustrates Coro.~ \ref{coro_y_max_corner} because the whole trajectory $\{d(\lambda) : \lambda\in[0,1]\}$ taken by the $d$-vector satisfies $ \frac{y_1}{(d(\lambda))_1}\geq\ldots\geq \frac{y_n}{(d(\lambda))_n} $ so maximality of $y$ (w.r.t.~classical majorization) is preserved throughout.
%%ARXIV ONLY
%
%

%% The Appendices part is started with the command \appendix;
%% appendix sections are then done as normal sections
%% \appendix

%% \section{}
%% \label{}
\appendix

%\section{Appendix}\label{appendix}
\renewcommand{\thesubsection}{\Alph{subsection}}

\section{Extreme Points of $s_d(3)$}\label{app_a}

\begin{lem}\label{lemma_extreme_points}
Let $d\in\mathbb R_{++}^3$ with $d_1>d_2>d_3$. 
\begin{itemize}
\item[(i)] If $d_1\geq d_2+d_3$, then the 10 extreme points of $s_d(3)$ are given by
\begin{align*}
\identity_3\quad\begin{pmatrix} 1&0&0\\0&1-\frac{d_3}{d_2}&1\\0&\frac{d_3}{d_2}&0 \end{pmatrix}\quad &\begin{pmatrix} 1-\frac{d_3}{d_1}&0&1\\0&1&0\\\frac{d_3}{d_1}&0&0 \end{pmatrix}\quad\begin{pmatrix} 1-\frac{d_2}{d_1}&1&0\\\frac{d_2-d_3}{d_1}&0&1\\\frac{d_3}{d_1}&0&0 \end{pmatrix}
\end{align*}${}$\vspace*{-16pt}
\begin{align*}
\begin{pmatrix} 1-\frac{d_3}{d_1}&\frac{d_3}{d_2}&0\\0&1-\frac{d_3}{d_2}&1\\\frac{d_3}{d_1}&0&0 \end{pmatrix}\quad&\begin{pmatrix} 1-\frac{d_2}{d_1}&1&0\\\frac{d_2}{d_1}&0&0\\0&0&1 \end{pmatrix}\quad\begin{pmatrix} 1-\frac{d_3}{d_1}&0&1\\\frac{d_3}{d_1}&1-\frac{d_3}{d_2}&0\\0&\frac{d_3}{d_2}&0 \end{pmatrix}
\end{align*}${}$\vspace*{-16pt}
\begin{align*}
\begin{pmatrix} 1-\frac{d_2-d_3}{d_1}&1-\frac{d_3}{d_2}&0\\\frac{d_2-d_3}{d_1}&0&1\\0&\frac{d_3}{d_2}&0 \end{pmatrix}\quad&\begin{pmatrix} 1-\frac{d_2}{d_1}&1-\frac{d_3}{d_2}&1\\\frac{d_2}{d_1}&0&0\\0&\frac{d_3}{d_2}&0 \end{pmatrix}\quad\begin{pmatrix} 1-\frac{d_2+d_3}{d_1}&1&1\\\frac{d_2}{d_1}&0&0\\\frac{d_3}{d_1}&0&0 \end{pmatrix}
\end{align*}
\item[(ii)] If $d_1< d_2+d_3$, then the 13 extreme points of $s_d(3)$ are given by
\begin{align*}
\identity_3\quad\begin{pmatrix} 1&0&0\\0&1-\frac{d_3}{d_2}&1\\0&\frac{d_3}{d_2}&0 \end{pmatrix}\quad &\begin{pmatrix} 1-\frac{d_3}{d_1}&0&1\\0&1&0\\\frac{d_3}{d_1}&0&0 \end{pmatrix}\quad\begin{pmatrix} 1-\frac{d_2}{d_1}&1&0\\\frac{d_2-d_3}{d_1}&0&1\\\frac{d_3}{d_1}&0&0 \end{pmatrix}
\end{align*}${}$\vspace*{-16pt}
\begin{align*}
\begin{pmatrix} 1-\frac{d_3}{d_1}&\frac{d_3}{d_2}&0\\0&1-\frac{d_3}{d_2}&1\\\frac{d_3}{d_1}&0&0 \end{pmatrix}\quad&\begin{pmatrix} 1-\frac{d_2}{d_1}&1&0\\\frac{d_2}{d_1}&0&0\\0&0&1 \end{pmatrix}\quad\begin{pmatrix} 1-\frac{d_3}{d_1}&0&1\\\frac{d_3}{d_1}&1-\frac{d_3}{d_2}&0\\0&\frac{d_3}{d_2}&0 \end{pmatrix}
\end{align*}${}$\vspace*{-16pt}
\begin{align*}
\begin{pmatrix} 1-\frac{d_2-d_3}{d_1}&1-\frac{d_3}{d_2}&0\\\frac{d_2-d_3}{d_1}&0&1\\0&\frac{d_3}{d_2}&0 \end{pmatrix}\quad&\begin{pmatrix} 1-\frac{d_2}{d_1}&1-\frac{d_3}{d_2}&1\\\frac{d_2}{d_1}&0&0\\0&\frac{d_3}{d_2}&0 \end{pmatrix}\quad\begin{pmatrix} 0&1&\frac{d_1-d_2}{d_3}\\\frac{d_2}{d_1}&0&0\\1-\frac{d_2}{d_1}&0&1-\frac{d_1-d_2}{d_3} \end{pmatrix}
\end{align*}${}$\vspace*{-16pt}
\begin{align*}
\begin{pmatrix} 0&\frac{d_1-d_3}{d_2}&1\\1-\frac{d_3}{d_1}&1-\frac{d_1-d_3}{d_2}&0\\\frac{d_3}{d_1}&0&0 \end{pmatrix}\quad&\begin{pmatrix} 0&\frac{d_1-d_3}{d_2}&1\\ \frac{d_2}{d_1} & 0 & 0 \\ 1-\frac{d_2}{d_1} &1-\frac{d_1-d_3}{d_2}&0 \end{pmatrix}\quad\begin{pmatrix} 0&1& \frac{d_1-d_2}{d_3} \\ 1-\frac{d_3}{d_1} &0& 1-\frac{d_1-d_2}{d_3}\\ \frac{d_3}{d_1} &0&0 \end{pmatrix}
\end{align*}
\end{itemize}
\end{lem}
\begin{proof}
The respective number of extreme points is stated in \cite[Remark 4.5]{Joe90} or, more recently, \cite[Ch.~IV]{Mazurek18}. Then one only has to verify that the above matrices (under the given assumptions) are in fact extremal in $s_d(3)$.
\end{proof}
Once we allow components of $d$ to coincide, the above extreme points simplify slightly (as already observed in \cite[Remark 4.5]{Joe90}). Within the setting of (i) if $d_2=d_3$ then one is left with 7 extreme points. For (ii) if either $d_1=d_2$ or $d_2=d_3$ then one has 10 extreme points and if $d_1=d_2=d_3$ then there are $6$ extreme points---namely the $3\times 3$ permutation matrices---which recovers Birkhoff's theorem, cf.~\cite[Thm.~2.A.2]{MarshallOlkin}. 

\section{Appendix to Section \ref{sec_prelim_vector}}\label{app_proofs}

In order to keep the main part of this paper sufficiently structured we outsourced some of the lengthy and technical lemmata.

\begin{lem}\label{lemma_minmax_matrix}
Let $A\in\{0,1\}^{m\times n}$ be a matrix such that $\operatorname{rank}(A)=n$,
$\unitvector^{\top}$ is a row of $A$, and for any two rows $a_1$, $a_2$ of $A$ their 
entrywise minimum $\min\{a_1,a_2\}$ and maximum $\max\{a_1,a_2\}$ are rows of 
$A$, as well. Then the following statements hold.
\begin{itemize}
\item[(i)] There exists a row $a$ of $A$ such that $a\unitvector=n-1$.%\marginpar{possibly extend to a 1-row as well;}
\item[(ii)] For every row $a$ of $A$ with $a\unitvector>1$ one finds a row $\tilde a$ of $A$ such that $\tilde a\unitvector=a\unitvector-1$ and $\tilde a\leq a$.%\marginpar{possibly extend to $a\unitvector<n$ one finds a larger vector}
\item[(iii)] There exist rows $a_1,\ldots,a_{n-1}$ of $A$ and a permutation $\sigma\in S_n$ such that
\begin{equation}\label{eq:aj_rows_perm}
\Lambda\,\underline{\sigma}={\footnotesize\begin{pmatrix}
a_1\\\vdots\\a_{n-1}\\\unitvector^{\top}
\end{pmatrix}}\,.
\end{equation}%\marginpar{possibly extend to: for \textbf{every} row $a$ of $A$ there exist $a_1,\ldots,a_{n-1}$ such that ...}
\end{itemize}
\end{lem}
\begin{proof}
(i): For all $j=1,\ldots,n$ define
$
S_j:=\{a : a\text{ is a row of }A\text{ and }ae_j=0\}\setminus\{0\}
$
as the collection of all (non-zero) rows of $A$ the $j$-th entry of which vanishes\footnote{
It may happen that $A$ contains (at most, due to rank condition) one column of ones so (at most) one of the $S_j$ might be empty, but one can still guarantee the existence of some $j$ such that $S_j\neq\emptyset$ (because $n\geq 2$, the case $n=1$ is trivial). %Hence for this proof we may speak of $S_j$ whenever it is non-empty.
}. Defining $a_j:=\max S_j$ this is a row of $A$ (due to the maximum property) with $a_j\unitvector\leq n-1$. It is obvious that any non-zero row $a$ of $A$ is in $S_j$ if and only if $a\leq a_j$---hence $a_j=a_k$ for any two $j,k$ implies $S_j=S_k$. Now there exists $k\in\{1,\ldots,n\}$ such that
\begin{equation}\label{eq:ak_aj_rel}
a_k\unitvector=\max_{j=1,\ldots,n}a_j\unitvector\quad(\geq 1)\,.
\end{equation}
If $a_k\unitvector=n-1$ then we are done. If $a_k\unitvector<n-1$ then one finds an index $i\neq k$ such that $a_ke_i=0$. Therefore $a_k\in S_i$ so $a_k\leq a_i$ but $a_k\unitvector\geq a_i\unitvector$ by \eqref{eq:ak_aj_rel}; this shows $a_k=a_i$ and thus $S_k=S_i$. We claim that this forces \textit{all} rows $a$ of $A$ to satisfy $a(e_i+e_k)\in\{0,2\}$; but then the linear span of all rows of $A$ has this property as well so it cannot contain $e_k$, contradicting $\operatorname{rank}(A)=n$. 

Indeed let $a$ be any row of $A$. If $a\in S_k=S_i$ then $ae_i=ae_k=0=ae_i+ae_k$. If $a\not\in S_k=S_i$ then $ae_i=ae_k=1$ so $ae_i+ae_k=2$.

(ii): We prove this via induction. The case $n=1$ is trivial. Now for $n\to n+1$ let $A\in\{0,1\}^{m\times(n+1)}$ with the above properties be given. Be aware of the following argument: For ${\tau}={\footnotesize\begin{pmatrix} 1&2&\cdots&j-1&j&j+1&\cdots&n+1\\2&3&\cdots&j&1&j+1&\cdots&n+1 \end{pmatrix}}$ and all $j=1,\ldots,n+1$ the matrix
\begin{align*}
A_j:=A\underline{\tau_j}{\footnotesize\begin{pmatrix} 0&\cdots&0\\1&&0\\&\ddots&\\0&&1 \end{pmatrix}}\in\mathbb R^{m\times n}
\end{align*}
is the original matrix $A$ but without the $j$-th column. Using this is easy to see that $\operatorname{rank}(A_j)=n$ (follows from, e.g., \cite[Thm.~0.4.5.(c)]{HJ1}), $\unitvector^{\top}\in A_j$, and the min-max condition for the rows of $A_j$ hold for all $j=1,\ldots,n+1$. Hence we may apply the induction hypothesis to any of these matrices $A_j$.

Now consider any row $a$ of $A$ with $a\unitvector>1$. There are two cases which, once verified, conclude the proof of (ii).
\begin{itemize}
\item[Case 1:] $a=\unitvector^{\top}$. By (i) we find $j\in\{1,\ldots,n+1\}$ such that $\unitvector^{\top}-e_j^\top $ is a row of $A$.
\item[Case 2:] $a\neq \unitvector^{\top}$ so there exists $j\in\{1,\ldots,n+1\}$ such that $ae_j=0$. Consider $A_j\in\mathbb R^{m\times n}$ and the truncated row $b$ of $A_j$ corresponding to $a$. By induction hypothesis ($b\unitvector=a\unitvector>1$) we find $\tilde b\in A_j$ such that $\tilde b\unitvector=b\unitvector-1$ and $\tilde b\leq b$. Now there exists a row $\hat a$ of $A$ which becomes $\tilde b$ when removing the $j$-th entry. Defining $\tilde a:=\min\{\hat a,a\}$ we know that this is a row of $A$ (min-max-property of $A$) and $\tilde a\leq a$ as well as $\tilde a\unitvector=\tilde b\unitvector =b\unitvector-1=a\unitvector-1$.
\end{itemize}

(iii): By assumption $\unitvector^{\top}\in A$ so using (ii) $A$ contains some $a_{n-1}$ of row sum $n-1$, which in turn yields $a_{n-2}\in A$ of row sum $n-2$ with $a_{n-2}\leq a_{n-1}$, and so forth. Eventually one ends up with rows $a_1,\ldots,a_{n-1}$ of $A$ which satisfy $a_j\unitvector=j$ for all $j=1,\ldots,n-1$ as well as $a_1\leq\ldots\leq a_{n-1}$; hence there exists a permutation ${\tau}\in S_n$ such that \eqref{eq:aj_rows_perm} holds.
\end{proof}
\begin{rmk}\label{rem_extend_lemma_minmax}
With an analogous argument one can show that every such matrix $A$ contains a standard basis vector as a row, and that for every row $a$ with $a\unitvector<n$ one finds $\tilde a\in A$ with $\tilde a\unitvector=a\unitvector+1$ and $\tilde a\geq a$. Therefore \textit{every} row of $A$ can be completed to a matrix of the form \eqref{eq:aj_rows_perm}.
\end{rmk}
\section{Appendix to Section \ref{sec_d_poly_analysis}}\label{app_proofs_B}

\begin{lem}\label{lemma_maj_sum_recovery}
Let $y\in\mathbb R^n$, $d\in\mathbb R_{++}^n$ be arbitrary and let $\sigma\in S_n$ satisfy\footnote{Such a permutation $\sigma$ always exists as it is just the decreasing ordering of the vector $\frac{y}{d}:=(\frac{y_i}{d_i})_{i=1}^n$.}
\begin{equation}\label{eq:sigma_ordering_yd}
\frac{y_{\sigma(1)}}{d_{\sigma(1)}}\geq \frac{y_{\sigma(2)}}{d_{\sigma(2)}}\geq\ldots\geq\frac{y_{\sigma(n)}}{d_{\sigma(n)}}\,.
\end{equation}
Then the following statements hold.
\begin{itemize}
\item[(i)] For all $c\in\mathbb R$, $k=1,\ldots,n$. 
$$
\unitvector^{\top}\Big( y-\frac{y_{\sigma(k)}}{d_{\sigma(k)}} d\Big)_++\frac{y_{\sigma(k)}}{d_{\sigma(k)}} c=\frac{y_{\sigma(1)}}{d_{\sigma(1)}} c-\sum_{i=1}^{k-1}\Big(\frac{y_{\sigma(i)}}{d_{\sigma(i)}}-\frac{y_{\sigma(i+1)}}{d_{\sigma(i+1)}}\Big)\Big( c-\sum_{j=1}^i d_{\sigma(j)} \Big)\,.
$$
\item[(ii)] Let $c\in(0,\unitvector^{\top}d]$ and let $k\in\{1,\ldots,n\}$ be the unique index such that $c-\sum_{i=1}^{k-1}d_{\sigma(i)}> 0$ but $c-\sum_{i=1}^k d_{\sigma(i)}\leq 0$. Then
$$
\min_{i=1,\ldots,n}\Big(\unitvector^{\top}\Big( y-\frac{y_{i}}{d_{i}} d\Big)_++\frac{y_{i}}{d_{i}} c\Big)=\Big(\sum\nolimits_{i=1}^{k-1}y_{\sigma(i)}\Big)+\frac{y_{\sigma(k)}}{d_{\sigma(k)}} \Big(c-\sum\nolimits_{i=1}^{k-1} d_{\sigma(i)}\Big) \,.
$$
\item[(iii)] If $d=\unitvector$, then $
\min_{i=1,\ldots,n} \unitvector^{\top}(y-y_i\unitvector)_++ky_i=\sum\nolimits_{i=1}^k y_{[i]}
$ for all $k=1,\ldots,n$.
\end{itemize}
\end{lem}
\begin{proof}
(i): This identity comes from
$$
\frac{y_{\sigma(1)}}{d_{\sigma(1)}} c-\sum_{i=1}^{k-1}\Big(\frac{y_{\sigma(i)}}{d_{\sigma(i)}}-\frac{y_{\sigma(i+1)}}{d_{\sigma(i+1)}}\Big)c=\frac{y_{\sigma(1)}}{d_{\sigma(1)}} c-\frac{y_{\sigma(1)}}{d_{\sigma(1)}} c+\frac{y_{\sigma(k)}}{d_{\sigma(k)}} c=\frac{y_{\sigma(k)}}{d_{\sigma(k)}} c
$$
as well as
\begin{align*}
\unitvector^{\top}\Big( y-\frac{y_{\sigma(k)}}{d_{\sigma(k)}} d\Big)_+&=\sum_{j=1}^{k-1}\Big(\frac{y_{\sigma(j)}}{d_{\sigma(j)}}-\frac{y_{\sigma(k)}}{d_{\sigma(k)}}\Big)d_{\sigma(j)}\\
&=\sum_{j=1}^{k-1}\sum_{i=j}^{k-1}\Big(\frac{y_{\sigma(i)}}{d_{\sigma(i)}}-\frac{y_{\sigma(i+1)}}{d_{\sigma(i+1)}}\Big)d_{\sigma(j)}=\sum_{i=1}^{k-1}\sum_{j=1}^{i}\Big(\frac{y_{\sigma(i)}}{d_{\sigma(i)}}-\frac{y_{\sigma(i+1)}}{d_{\sigma(i+1)}}\Big)d_{\sigma(j)}
\end{align*}
where in the last step one just re-enumerates the index set $\{(i,j) : 1\leq j\leq i\leq k-1\}$. 

(ii): Using (i)
\begin{align*}
\min_{i=1,\ldots,n}\unitvector^{\top}\Big( y-\frac{y_{i}}{d_{i}} d\Big)_++\frac{y_{i}}{d_{i}} c&=\min_{\ell=1,\ldots,n}\unitvector^{\top}\Big( y-\frac{y_{\sigma(\ell)}}{d_{\sigma(\ell)}} d\Big)_++\frac{y_{\sigma(\ell)}}{d_{\sigma(\ell)}} c\\
%&=\min_{\ell=1,\ldots,n}\frac{y_{\sigma(1)}}{d_{\sigma(1)}} c-\sum_{i=1}^{\ell-1}\Big(\frac{y_{\sigma(i)}}{d_{\sigma(i)}}-\frac{y_{\sigma(i+1)}}{d_{\sigma(i+1)}}\Big)\Big( c-\sum_{j=1}^i d_{\sigma(j)} \Big)\\
&=\frac{y_{\sigma(1)}}{d_{\sigma(1)}} c-\max_{\ell=1,\ldots,n}\sum_{i=1}^{\ell-1}\Big(\frac{y_{\sigma(i)}}{d_{\sigma(i)}}-\frac{y_{\sigma(i+1)}}{d_{\sigma(i+1)}}\Big)\Big( c-\sum_{j=1}^i d_{\sigma(j)} \Big)\,.
\end{align*}
There are two important things to notice here: the expression $\frac{y_{\sigma(i)}}{d_{\sigma(i)}}-\frac{y_{\sigma(i+1)}}{d_{\sigma(i+1)}}$ is always non-negative by \eqref{eq:sigma_ordering_yd} and, moreover, the map 
$
g:\{0,\ldots,n\}\to \mathbb R$, $i\mapsto c-\sum\nolimits_{j=1}^i d_{\sigma(j)}
$
satisfies $g(0)=c>0$, $g(n)=c-\unitvector^{\top}d\leq 0$, and is strictly monotonically decreasing. Thus the index $k$ described above exists, is unique, and we get
$$
\max_{\ell=1,\ldots,n}\sum_{i=1}^{\ell-1}\Big(\frac{y_{\sigma(i)}}{d_{\sigma(i)}}-\frac{y_{\sigma(i+1)}}{d_{\sigma(i+1)}}\Big)\Big( c-\sum_{j=1}^i d_{\sigma(j)} \Big)=\sum_{i=1}^{k-1}\Big(\frac{y_{\sigma(i)}}{d_{\sigma(i)}}-\frac{y_{\sigma(i+1)}}{d_{\sigma(i+1)}}\Big)\Big( c-\sum_{j=1}^i d_{\sigma(j)} \Big)
$$
as we simply disregard all negative summands ($i\geq k$). Therefore
\begin{align*}
\min_{i=1,\ldots,n}\unitvector^{\top}\Big( y-\frac{y_{i}}{d_{i}} d\Big)_++\frac{y_{i}}{d_{i}} c
&=\frac{y_{\sigma(1)}}{d_{\sigma(1)}} c-\sum_{i=1}^{k-1}\Big(\frac{y_{\sigma(i)}}{d_{\sigma(i)}}-\frac{y_{\sigma(i+1)}}{d_{\sigma(i+1)}}\Big)\Big( c-\sum_{j=1}^i d_{\sigma(j)} \Big)\\
&\overset{\text{(i)}}=\unitvector^{\top}\Big( y-\frac{y_{\sigma(k)}}{d_{\sigma(k)}} d\Big)_++\frac{y_{\sigma(k)}}{d_{\sigma(k)}} c\\
&= \sum\nolimits_{i=1}^{k-1}d_{\sigma(i)} \Big( \frac{y_{\sigma(i)}}{d_{\sigma(i)}}-\frac{y_{\sigma(k)}}{d_{\sigma(k)}} \Big)+\frac{y_{\sigma(k)}}{d_{\sigma(k)}} c \\
&= \Big(\sum\nolimits_{i=1}^{k-1}y_{\sigma(i)}\Big)+\frac{y_{\sigma(k)}}{d_{\sigma(k)}} \Big(c-\sum\nolimits_{i=1}^{k-1} d_{\sigma(i)}\Big)\,.
\end{align*}

(iii): Direct consequence of (ii).
%For all $k=1,\ldots,n$
%\begin{align*}
%\min_{i=1,\ldots,n}\unitvector^{\top}\big( y-y_{[i]} d\big)_++y_{[i]} k\overset{\text{(ii)}}=\unitvector^{\top}(y-y_{[k]} \unitvector)_++ky_{[k]}=\sum\nolimits_{i=1}^ky_{[i]}\tag*{\qedhere}
%\end{align*}
%OLD PROOF
%Considering the second identity, for all $i,k=1,\ldots,n$
%\begin{align*}
%\unitvector^{\top}(y-y_{[i]} \unitvector)_++ky_{[i]}=\sum\nolimits_{j=1}^i(y_{[j]}-y_{[i]})+ky_{[i]}=\sum\nolimits_{j=1}^iy_{[j]} +(k-i)y_{[i]}\,.
%\end{align*}
%On the one hand $\unitvector^{\top}(y-y_{[k]} \unitvector)_++ky_{[k]}=\sum_{j=1}^ky_{[j]}$ and on the other
%\begin{align*}
%i\leq k\ &\Rightarrow\ \sum\nolimits_{j=1}^i y_{[j]} +\underbrace{(k-i)}_{\geq 0}y_{[i]}\geq \sum\nolimits_{j=1}^iy_{[l]}+\sum\nolimits_{j=i+1}^k y_{[i]} =\sum\nolimits_{j=1}^k y_{[j]}\\
%i\geq k\ &\Rightarrow\ \sum\nolimits_{j=1}^i y_{[j]} +\underbrace{(k-i)}_{\leq 0}y_{[i]}\geq \sum\nolimits_{j=1}^iy_{[l]}-\sum\nolimits_{j=k+1}^i y_{[i]} =\sum\nolimits_{j=1}^k y_{[j]}\,.\tag*{\qedhere}
%\end{align*}
\end{proof}
\begin{lem}\label{lemma_d_ordering_maj}
Let $d\in\mathbb R_{++}^n$, $k=1,\ldots,n-1$, $\tau\in S_n$, and pairwise different numbers $\alpha_1,\ldots,\alpha_k\in\{1,\ldots,n\}$ be given. Then
$$
v:=\begin{pmatrix} \sum_{i=1}^{\alpha_1-1}d_{\tau(i)}\\\vdots\\\sum_{i=1}^{\alpha_k-1}d_{\tau(i)}\\\sum_{i=1}^kd_{\tau(\alpha_i)} \end{pmatrix}\prec\begin{pmatrix} \sum_{i=1}^{\alpha_1}d_{\tau(i)}\\\vdots\\\sum_{i=1}^{\alpha_k}d_{\tau(i)}\\0 \end{pmatrix}=:w\in\mathbb R^{k+1}\,.
$$
\end{lem}
\begin{proof}
W.l.o.g.~$\alpha_1>\ldots>\alpha_k$---reordering the $\alpha_i$ amounts to reordering $v,w$ but classical majorization is permutation invariant. By definition $v\prec w$ holds iff $\unitvector^{\top}v=\unitvector^{\top}w$ and together with $\sum_{i=1}^\ell v_{[i]}\leq\sum_{i=1}^\ell w_{[i]}$ for all $\ell=1,\ldots,k$ (the former is readily verified). Because the $\alpha_i$ are ordered decreasingly one finds unique $\xi\in\{1,\ldots,k+1\}$ such that
$$
\sum\nolimits_{i=1}^{\alpha_\xi-1}d_{\tau(i)}< \sum\nolimits_{i=1}^kd_{\tau(\alpha_i)}\leq \sum\nolimits_{i=1}^{\alpha_{\xi-1}-1}d_{\tau(i)}
$$
(where $\alpha_0:=n+1$ and $\alpha_{k+1}:=0$). Thus $v\prec w$ is equivalent to
$$
[v]= \begin{pmatrix} \sum_{i=1}^{\alpha_1-1}d_{\tau(i)}\\\vdots\\ \sum_{i=1}^{\alpha_{\xi-1}-1}d_{\tau(i)} \\ \sum_{i=1}^kd_{\tau(\alpha_i)} \\ \sum_{i=1}^{\alpha_\xi-1}d_{\tau(i)}\\\vdots \\\sum_{i=1}^{\alpha_{k-1}-1}d_{\tau(i)} \\\sum_{i=1}^{\alpha_{k}-1}d_{\tau(i)} \end{pmatrix} 
\prec
\begin{pmatrix} \sum_{i=1}^{\alpha_1}d_{\tau(i)}\\\vdots\\ \sum_{i=1}^{\alpha_{\xi-1}}d_{\tau(i)}\\ \sum_{i=1}^{\alpha_\xi}d_{\tau(i)} \\ \sum_{i=1}^{\alpha_{\xi+1}}d_{\tau(i)}\\\vdots\\ \sum_{i=1}^{\alpha_k}d_{\tau(i)}\\0 \end{pmatrix} =[w]\,.
$$
The first $\xi-1$ partial sum conditions are evident.
% (because $v_{[j]}\leq w_{[j]}$ for all $j=1,\ldots,\xi-1$ individually). 
Now consider $\ell\in\{\xi,\ldots,k\}$. Then
\begin{align*}
\sum_{j=1}^\ell w_{[j]}-\sum_{j=1}^\ell v_{[j]}&=\sum_{j=1}^\ell \Big(\sum_{i=1}^{\alpha_j}d_{\tau(i)}\Big)-\Big(\sum_{j=1}^{\ell-1}\Big(\sum_{i=1}^{\alpha_j-1}d_{\tau(i)}\Big)+\sum_{i=1}^kd_{\tau(\alpha_i)}\Big)\\
&= \sum_{i=1}^{\alpha_\ell}d_{\tau(i)}+\sum_{j=1}^{\ell-1}d_{\tau(\alpha_j)}-\sum_{i=1}^kd_{\tau(\alpha_i)}= \sum_{i=1}^{\alpha_\ell}d_{\tau(i)}-\sum_{i=\ell}^kd_{\tau(\alpha_i)}\geq0
\end{align*}
where in the last step we used that the entries of $d$ are non-negative and, more importantly, that $\{\alpha_{k},\alpha_{k-1},\ldots,\alpha_{\ell+1},\alpha_{\ell}\}\subseteq\{1,2,\ldots,\alpha_{\ell}-1,\alpha_{\ell}\}$ due to the ordering of the $\alpha_i$
\end{proof}

\begin{lem}\label{lemma_minmax}
Let $m,n\in\mathbb N$ and $x\in\mathbb R^n$, $y\in\mathbb R_+^n$, $z\in\mathbb R^m$ be given. Then
$$
\max_{k=1,\ldots,m}\min_{i=1,\ldots,n}(x_i+y_iz_k)=\min_{i=1,\ldots,n}\big(x_i+y_i\big(\max_{k=1,\ldots,m} z_k\big)\big)\,.
$$
\end{lem}
\begin{proof}
Direct computation:
\begin{align*}
\max_{k=1,\ldots,m}\min_{i=1,\ldots,n}(x_i+y_iz_k)&\leq\min_{i=1,\ldots,n}\max_{k=1,\ldots,m}(x_i+y_iz_k)\\
&=\min_{i=1,\ldots,n}\big(x_i+y_i\big(\max_{k=1,\ldots,m} z_k\big)\big)\leq \max_{k=1,\ldots,m}\min_{i=1,\ldots,n}(x_i+y_iz_k)
\end{align*}
The first step is the usual max-min inequality, the second step works because $y\geq 0$, and in the final step we use that
$
\min_{i=1,\ldots,n}\big(x_i+y_iz_l)\leq \max_{k=1,\ldots,m}\min_{i=1,\ldots,n}(x_i+y_iz_k)
$
for all $l=1,\ldots,m$---which in particular holds for the index $l$ which satisfies $z_l=\max_{k=1,\ldots,m} z_k$.
\end{proof}
\section{Examples and Counterexamples}\label{app_b}
\begin{example}\label{ex_proof_ext_point_fail}
Let $n=4$ so
$$
M={\footnotesize\begin{pmatrix} 1&0&0&0\\
0&1&0&0\\
0&0&1&0\\
0&0&0&1\\
1&1&0&0\\
1&0&1&0\\
1&0&0&1\\
0&1&1&0\\
0&1&0&1\\
0&0&1&1\\
1&1&1&0\\
1&1&0&1\\
1&0&1&1\\
0&1&1&1\\
1&1&1&1\\
-1&-1&-1&-1 \end{pmatrix}}
\qquad\text{ and choose }\qquad
b={\footnotesize\begin{pmatrix} 0\\
0\\
0\\
0\\
0\\
-1/2\\
-1/4\\
0\\
0\\
0\\
-1/2\\
-1/2\\
-5/8\\
0\\
-1\\
1 \end{pmatrix}}\,.
$$
By Definition \ref{def_Ebsigma} and Lemma \ref{lemma_E_b_sigma_properties}
\begin{align*}
\{E_b(\sigma) : \sigma\in S_4\}=\Big\{& {\footnotesize
\begin{pmatrix} 0 \\0 \\-1/2 \\-1/2 \end{pmatrix},
\begin{pmatrix} 0 \\ -3/8 \\ -1/2 \\ -1/8 \end{pmatrix},
\begin{pmatrix} 0 \\ -1/4 \\ -1/2 \\ -1/4 \end{pmatrix}, 
\begin{pmatrix} 0 \\ -3/8 \\ -3/8 \\ -1/4 \end{pmatrix}, 
\begin{pmatrix} -1/2 \\ 0 \\ 0 \\ -1/2 \end{pmatrix}, 
\begin{pmatrix} -1 \\ 0 \\ 0 \\ 0 \end{pmatrix}, }\\
&{\footnotesize
\begin{pmatrix} -1/2 \\ 0 \\ -1/2 \\ 0 \end{pmatrix}, 
\begin{pmatrix} -1/2 \\ -3/8 \\ 0 \\ -1/8 \end{pmatrix}, 
\begin{pmatrix} -5/8 \\ -3/8 \\ 0 \\ 0 \end{pmatrix}, 
\begin{pmatrix} -1/4 \\ -1/4 \\ -1/2 \\ 0 \end{pmatrix}, 
\begin{pmatrix} -1/4 \\ -3/8 \\ -3/8 \\ 0 \end{pmatrix}
}\Big\}\,.
\end{align*}
%All of those vectors aside from the second and the fourth one are in $\{x\in\mathbb R^4 : Mx\leq b\}$. 
The second and the fourth vector from this list are the solutions to
\begin{equation*}%\label{eq:extreme_out}
{\footnotesize\begin{pmatrix} 1&0&0&0\\1&0&1&0\\1&0&1&1\\1&1&1&1 \end{pmatrix}}p={\footnotesize\begin{pmatrix} 0\\-1/2\\-5/8\\-1 \end{pmatrix}}\quad\text{ and }\quad
{\footnotesize\begin{pmatrix} 1&0&0&0\\1&0&0&1\\1&0&1&1\\1&1&1&1 \end{pmatrix}}p={\footnotesize\begin{pmatrix}0\\-1/4\\-5/8\\-1 \end{pmatrix}}\,,
\end{equation*}
respectively and are not in $\{x\in\mathbb R^4 : Mx\leq b\}$---but every other point of $\{E_b(\sigma) : \sigma\in S_4\}$ is in. On the other hand one readily verifies that $p=-\frac{1}{8}(1,3,3,1)^\top $ satisfies $Mp\leq b$ and solves
\begin{equation*}%\label{eq:extreme_in}
{\footnotesize\begin{pmatrix} 1&0&1&0\\1&0&0&1\\1&0&1&1\\1&1&1&1 \end{pmatrix}}p={\footnotesize\begin{pmatrix} -1/2\\-1/4\\-5/8\\-1 \end{pmatrix}}
\end{equation*}
so it is extreme in $\{x\in\mathbb R^4 : Mx\leq b\}$ by Lemma \ref{lemma_extreme_points_h_descr} but $p\not\in\{E_b(\sigma) : \sigma\in S_4\}$. Thus there exist extreme points of $\{x\in\mathbb R^4 : Mx\leq b\}$ not of the form $E_b(\sigma)$.
% which is in accordance with Thm.~\ref{thm_Ebsigma_extreme}. Note that the submatrices of $M$ (of full rank) in \eqref{eq:extreme_out} and \eqref{eq:extreme_in} strongly resemble each other.
%Note that $P$ is active because $(-1,0,0,0)^\top ,(0,0,-1/2,-1/2)^\top ,p\in P$ so $\min_{y\in P}(b-My)=0$ as is readily verified. Thus---unlike in Example ??---requiring activeness does neither prohibit nor fix this phenomenon. 
\end{example}

\begin{example}\label{example_1}
Let $n=3$ and $d=\unitvector$ (so $\prec_d$ becomes $\prec$). Consider the probability vectors $x=\frac15(2,1,2)^\top $, $y=\frac14(1,2,1)^\top $, and
%$$
%x=\frac15\begin{pmatrix} 2\\1\\2 \end{pmatrix}=\begin{pmatrix} 0.4\\0.2\\0.4 \end{pmatrix}\qquad\text{ and }\qquad y=\frac14\begin{pmatrix} 1\\2\\1 \end{pmatrix}=\begin{pmatrix} 0.25\\0.5\\0.25 \end{pmatrix}
%$$
their joining line segment $P:=\operatorname{conv}\lbrace x,y\rbrace$. %, i.e.~their convex hull.
Be aware that $P$ as well as $M_{\unitvector}(P)=\bigcup_{z\in P}\lbrace v\in\mathbb R^n : v\prec z\rbrace$ are subsets of $\Delta^2$% and that $M_{\unitvector}$ (much like $M_d$) is a closure operator (refer to Thm.~\ref{lemma_R_d_closed})
. %Using this 
One readily verifies
\begin{equation}\label{eq:maj_decomp}
M_{\unitvector}(P)=\lbrace v\in\mathbb R^n : v\prec x\,\vee\,v\prec y\rbrace=M_{\unitvector}(x)\cup M_{\unitvector}(y)\,,
\end{equation}
refer also to Fig.~\ref{Abb1}. Now although $x, \tilde y:=\frac14(1,1,2)^\top $ are in $M_{\unitvector}(P)$ one has
$$
\frac12 x+\frac12\tilde y=\frac1{40}\begin{pmatrix} 13\\9\\18 \end{pmatrix} =\begin{pmatrix} 0.325\\0.225\\0.45 \end{pmatrix} \overset{\eqref{eq:maj_decomp}}{\not\in} M_{\unitvector}(P)
$$
as it is neither majorized by $x$ nor by $y$; therefore $M_{\unitvector}(P)$ is not convex.
\end{example}
\begin{figure}[!htb]
\centering
\includegraphics[width=0.49\textwidth]{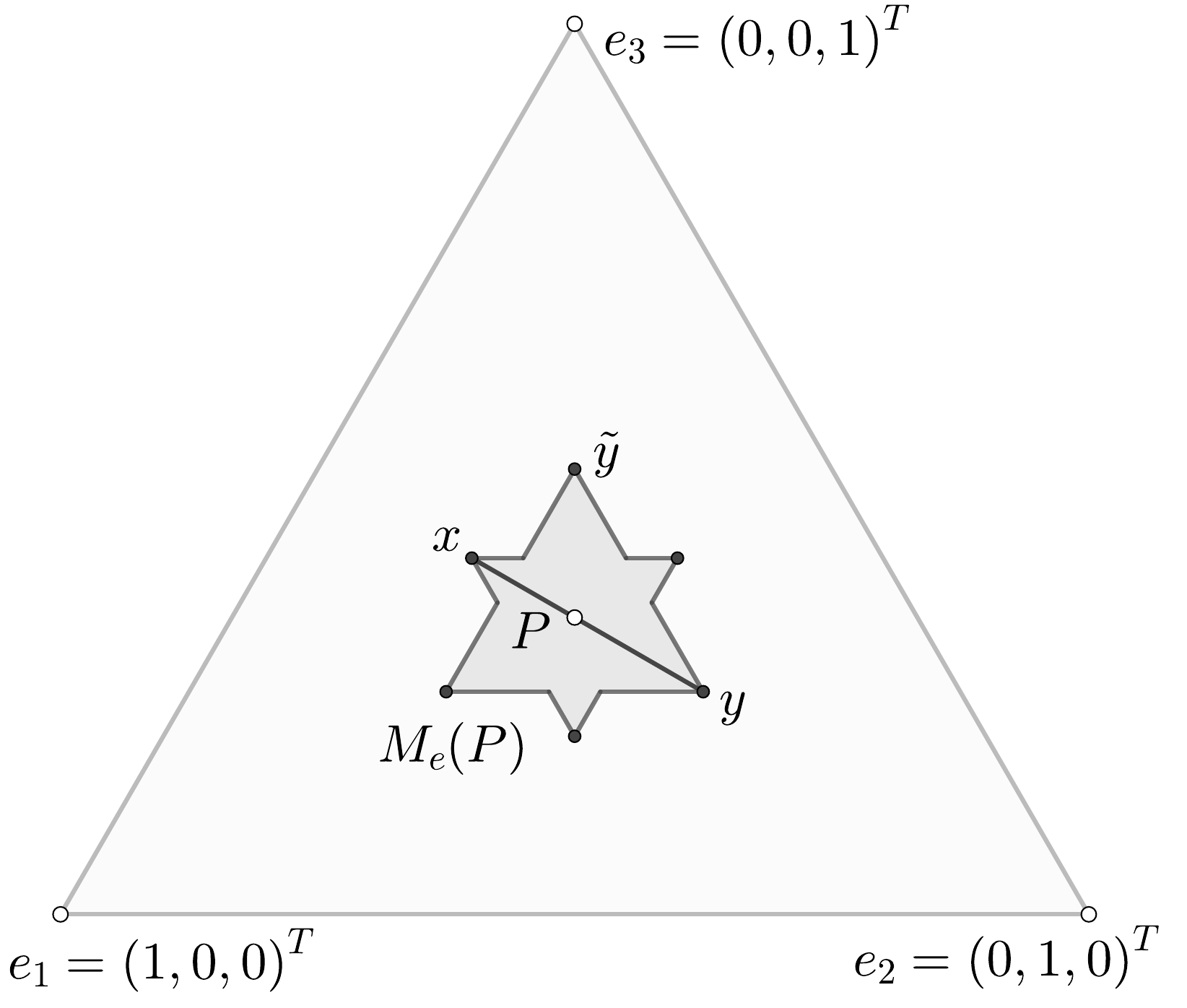}
\includegraphics[width=0.49\textwidth]{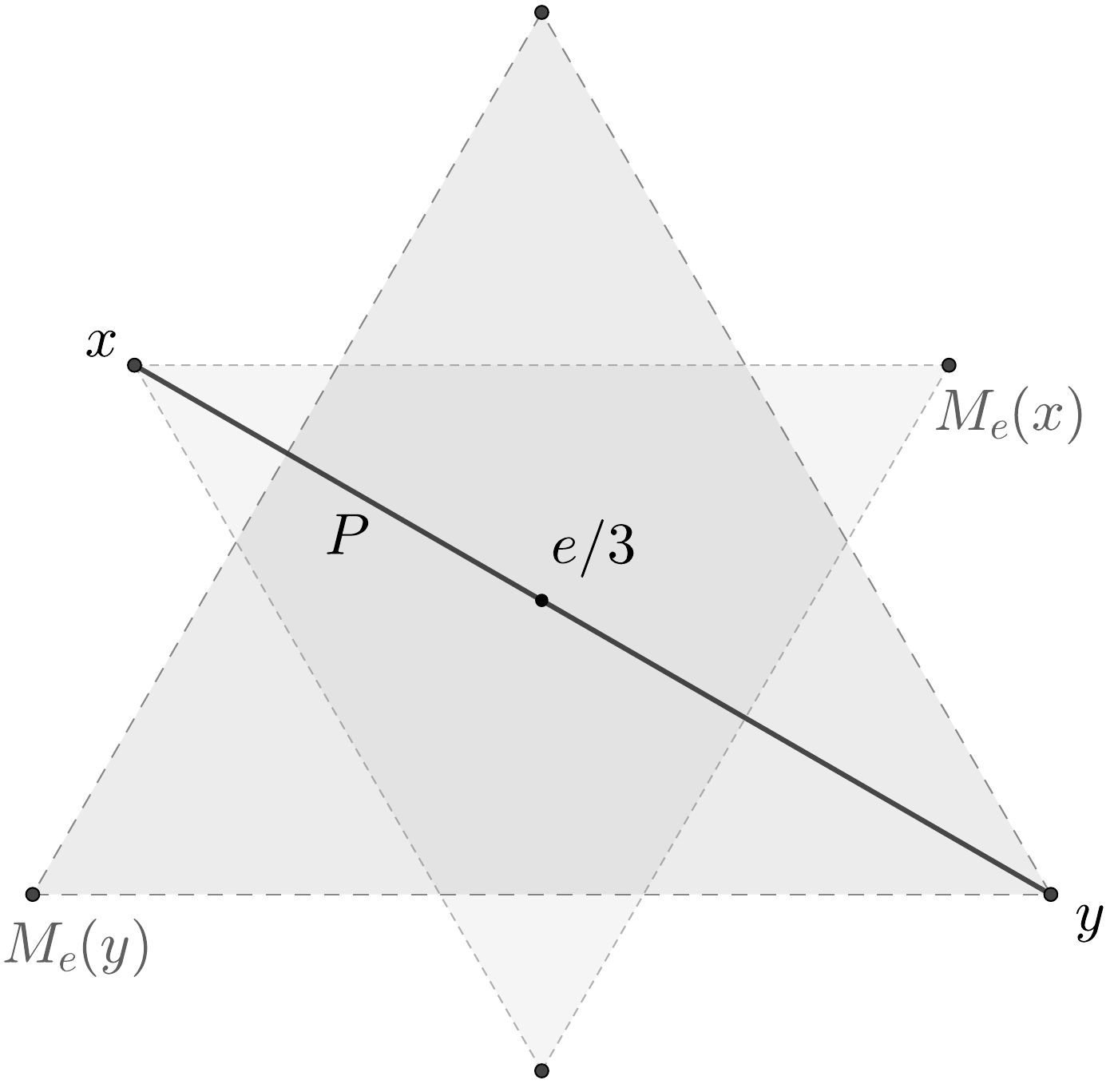}
\caption{%\textit{(Color online).} 
Visualization of Example \ref{example_1} on the 3-dimensional standard simplex. 
The image on the right zooms in on $M_{\unitvector}(P)$ and shows the decomposition into $M_{\unitvector}(x)$ 
and $M_{\unitvector}(y)$. In particular one sees that for all $z\in P$ one has either $z\prec x$ ($\Leftrightarrow z\in M_{\unitvector}(x)$) 
or $z\prec y$ ($\Leftrightarrow z\in M_{\unitvector}(y)$) which implies \eqref{eq:maj_decomp}.}\label{Abb1}
\end{figure}

%\begin{example}\label{example_app_proof1}
%Consider $A_1:=\{01101,11111\}$ and an associated matrix of full rank
%$$
%X:={\scriptsize \begin{pmatrix} 1&1&0&0&0\\1&0&0&0&1\\1&1&0&0&1\\0&1&1&0&1\\1&1&1&1&1 \end{pmatrix}=\begin{pmatrix} v_1\\v_2\\v_3\\m_1\\\unitvector^{\top} \end{pmatrix} }\in\{0,1\}^{n\times n}\,.
%$$
%Obviously there is no $\breve A\in S_X$ with $A_1\subsetneq\breve A$ (so $A_1$ does not have a ``permutation-conformal extension'' within $X$) but $\{10001,11001,11111\},\{11000,11001,11111\}\in S_X$ so there exist elements of $S_X$ with strictly larger cardinality than $A_1$. Indeed
%$$
%\mathfrak f(A_1)=\{01000,00100,00001,01100,01001,00101,11101,01111\}
%$$
%but 
%$
%\mathfrak f(A_1)X^{-1}\cap (\mathbb Q_+^3\times\mathbb Q\times \mathbb Q_-)=\emptyset
%$
%where
%$$
%X^{-1}={\scriptsize \begin{pmatrix} 1&1&-1&0&0\\0&-1&1&0&0\\1&1&-2&1&0\\-1&-1&1&-1&1\\-1&0&1&0&0 \end{pmatrix} }
%$$
%by direct calculation so there exist no coefficients $q_1^v,q_2^v,q_3^v,q_\unitvector\in\mathbb Q_+$, $q_m\in\mathbb Q$ such that
%$$
%q_1^vv_1+q_2^vv_2+q^3_vv_3+q_mm_1-q_\unitvector\unitvector^{\top}\in\mathfrak f(A_1)
%$$
%so Lemma \ref{lemma_perm_char} (e) possibly fails when relaxing its assumptions in this manner.
%\end{example}

\begin{example}\label{ex_discont_IR_plus}
Consider $y=(1,1,1)^\top $, $\lambda\in[0,\frac12]$, and $d(\lambda)=(1,\lambda,\lambda^2)$. For all $\lambda\in(0,\frac12]$---using Thm.~\ref{thm_maj_halfspace} \& \ref{thm_Eb_sigma}---by direct computation one obtains
%(cf.~also Example \ref{example_2})
\begin{align*}
M_{d(\lambda)}(y)&=\big\{x\in\mathbb R^3 : Mx\leq \big(3-\lambda-\lambda^2 , 2-\lambda , 1 , 3-\lambda^2 , 3-\lambda , 2 , 3 , -3\big)^\top \big\}\\
&=\operatorname{conv}\Big\{{\footnotesize
\begin{pmatrix} 3-\lambda-\lambda^2\\\lambda\\\lambda^2 \end{pmatrix},
\begin{pmatrix} 1+\lambda-\lambda^2\\2-\lambda\\\lambda^2 \end{pmatrix},
\begin{pmatrix} 1\\2-\lambda\\\lambda\end{pmatrix},
\begin{pmatrix} 2-\lambda\\\lambda\\1 \end{pmatrix},
\begin{pmatrix} 1\\1\\1 \end{pmatrix}
}
\Big\}\,.
\end{align*}
as well as\footnote{
For $\lambda=0$, i.e.~$d=d(0)=(1,0,0)$ it is easy so see that every $d$-stochastic matrix is of the form
$$
A=\begin{pmatrix} 1&&\\0&v&w\\0&& \end{pmatrix}\quad\text{ with arbitrary }v,w\in\Delta^2\,.
$$
Thus $M_{d(0)}(y)=\{Ay : A\in s_{d(0)}(3)\}=\{(1+v_1+w_1,v_2+w_2,v_3+w_3)^\top : v,w\in\Delta^2\}$ which has extreme points $(3,0,0)^\top , (1,2,0)^\top , (1,0,2)^\top $.
}
\begin{align*}
M_{d(\lambda)}(y)\overset{\lambda\to 0^+}\to\operatorname{conv}\Big\{{\footnotesize
\begin{pmatrix} 3\\0\\0 \end{pmatrix},
\begin{pmatrix} 1\\2\\0 \end{pmatrix},
\begin{pmatrix} 2\\0\\1 \end{pmatrix},
\begin{pmatrix} 1\\1\\1 \end{pmatrix}
}
\Big\}\neq \operatorname{conv}\Big\{{\footnotesize
\begin{pmatrix} 3\\0\\0 \end{pmatrix},
\begin{pmatrix} 1\\2\\0 \end{pmatrix},
\begin{pmatrix} 1\\0\\2 \end{pmatrix}
}
\Big\}=M_{d(0)}(y)\,.
\end{align*}
\end{example}

\begin{example}\label{ex_pos_y_necess}
Let $y=(1,1,-1)^\top $, $d=(1,2,3)^\top $. Using Thm.~\ref{thm_maj_halfspace} \& \ref{thm_Eb_sigma}, by direct computation
\begin{align*}
M_d(y)&=\Big\{x\in\mathbb R^3 : Mx\leq \frac16\big(6,9,12,12,10,8,6,-6\big)^\top \Big\}\\
&=\operatorname{conv}\Big\{{\footnotesize
\begin{pmatrix} 1\\1\\-1 \end{pmatrix},
\begin{pmatrix} 1\\-2/3\\2/3 \end{pmatrix},
\begin{pmatrix} 1/2\\3/2\\-1 \end{pmatrix},
\begin{pmatrix} -1/3\\3/2\\-1/6 \end{pmatrix},
\begin{pmatrix} -1/3\\-2/3\\2 \end{pmatrix}}
\big\}
\end{align*}
so the only possible candidate for the point $z$ from Thm.~\ref{theorem_max_corner_maj} is the vector $z=(-1/3,2/3,2)$ (because $z_{[1]}=2> x_{[1]}$ for all other extreme points $x$). However, one has $ y\not\prec z$ because
$
y_{[1]}+y_{[2]}=1+1=2 > \frac{5}{3} = z_{[1]}+z_{[2]}
.$
\end{example}

\begin{example}\label{ex_counter_y_max}
Let $y=(4,0,1)^\top $, $d=(4,2,1)^\top $. Using Thm.~\ref{thm_maj_halfspace} \& \ref{thm_Eb_sigma}, by direct computation
\begin{align*}
M_d(y)=\big\{x\in\mathbb R^3 : Mx\leq (4,2,1,5,5,3,5,-5)^\top \big\}=\operatorname{conv}\Big\{{\footnotesize
\begin{pmatrix} 4\\ 1 \\0 \end{pmatrix},
\begin{pmatrix} 4\\ 0 \\1 \end{pmatrix},
\begin{pmatrix}3 \\2\\0 \end{pmatrix},
\begin{pmatrix} 2\\2\\1 \end{pmatrix}}
\big\}\,.
\end{align*}
Therefore $M_d(y)\subseteq M_{\unitvector}(y)$, but $\frac{y}{d}=(1,0,1)^\top $ and $d$ are not similarly ordered. Be aware that this example generalizes to $y=(\alpha^2,0,1)^\top $, $d=(\alpha^2,\alpha,1)^\top $ for all $\alpha>1$ so the phenomenon occurs no matter whether $d_1\geq d_2+d_3$ or $d_1<d_2+d_3$ (cf.~\ref{app_a}).
\end{example}

%
%
%%ARXIV ONLY
\begin{example}\label{example_not_cont_A}
To see discontinuity of the map
$$
P(b):D(P)\to \mathcal P_c(\mathbb R^n)\qquad A\mapsto P_A(b)=\{x\in\mathbb R^n : Ax\leq b\}
$$
for arbitrary but fix $b\in\mathbb R^m$ and domain\footnote{This choice of domain ensures that the co-domain of $P$ is $P_c(\mathbb R^n)$, i.e.~that all $P_A(b)$ are non-empty and bounded (hence compact), cf.~\cite[Ch.~8.2]{Schrijver86}.} $D(P)$ consisting of all $A\in\mathbb R^{m\times n}$ such that $P_A(0)=\{0\}$ and $P_A(b)\neq\emptyset$, consider the following: let
$$
A={\footnotesize\begin{pmatrix}
1&0\\-1&0\\0&1\\0&-1
\end{pmatrix}}\ ,\ A_t={\footnotesize\begin{pmatrix}
1&0\\-1&0\\\sin(t)&\cos(t)\\0&-1
\end{pmatrix}}\qquad\text{ and }\qquad b={\footnotesize\begin{pmatrix}
1\\1\\0\\0
\end{pmatrix}}
$$
for all $t\in(0,1]$. It is readily verified that
\begin{align*}
P_{A}(b)&=\operatorname{conv}\big\{{\footnotesize \begin{pmatrix} -1\\0 \end{pmatrix},\begin{pmatrix} 1\\0 \end{pmatrix} }\big\}\hspace*{64pt}\text{ as well as}\\
P_{A_t}(b)&=\operatorname{conv}\big\{{\footnotesize \begin{pmatrix} 0\\0 \end{pmatrix},\begin{pmatrix} -1\\0 \end{pmatrix},\begin{pmatrix} -1\\\tan(t) \end{pmatrix} }\big\}\qquad\text{ for all }t>0
\end{align*}
and $(A_t)_{t\geq 0}\subset D(P)$. Thus by definition of the Hausdorff metric
$$
\delta(P_{A_t}(b),P_A(b))\geq \max_{z\in P_A(b)}\min_{w\in P_{A_t}(b)}\|z-w\|_1\geq \min_{w\in P_{A_t}(b)}\big\| {\footnotesize \begin{pmatrix} 1\\0 \end{pmatrix}}-w\big\|_1=1
$$
for all $t>0$ but, obviously, $\lim_{t\to 0^+}\|A_t-A\|=0$ so $P(b)$ cannot be continuous. 
\end{example}
%%ARXIV ONLY
%
%
%
%
%%ARXIV ONLY
\begin{example}\label{ex_wandering_d_vector}
Let $y=(3,2,1)^\top $, $\lambda\in[0,1]$ and $d(\lambda)=(2+\lambda,2,2-\lambda)$ so 
$$
M_{d(0)}(y)=M_\unitvector(y)=\operatorname{conv}\{\underline{\sigma}y : \sigma\in S_3\}\qquad\text{ and }\qquad M_{d(1)}(y)=M_y(y)=\{y\}
$$
(cf.~also Example \ref{example_2}). Thus the parameter $\lambda\in[0,1]$ describes the deformation of a classical majorization polytope into a singleton. Indeed one readily computes
\begin{align*}
M_{d(\lambda)}(y)=\Big\{x\in\mathbb R^3 : Mx\leq {\footnotesize\begin{pmatrix}
3\\\frac{6}{2+\lambda}\\\frac{6-3\lambda}{2+\lambda}\\5\\5-\lambda\\5-2\lambda\\6\\-6
\end{pmatrix}}\Big\}=\operatorname{conv}\Big\{{\footnotesize
\begin{pmatrix} 3\\2\\1 \end{pmatrix},
\begin{pmatrix} 3\\1+\lambda\\2-\lambda \end{pmatrix},
\frac{1}{2+\lambda}\begin{pmatrix} 4+5\lambda\\6\\2+\lambda \end{pmatrix},}&\\
{\footnotesize\frac{1}{2+\lambda}\begin{pmatrix} 2\lambda^2+5\lambda+2\\6\\-2\lambda^2+\lambda+4 \end{pmatrix},
\frac{1}{2+\lambda}\begin{pmatrix} -\lambda^2+6\lambda+4\\\lambda^2+3\lambda+2\\6-3\lambda \end{pmatrix},
\frac{1}{2+\lambda}\begin{pmatrix} 2\lambda^2+5\lambda+2\\-2\lambda^2+4\lambda+4\\6-3\lambda \end{pmatrix}}&
\Big\}\,.
\end{align*}
\end{example}
\begin{figure}[!htb]
\centering
(i)\includegraphics[width=0.45\textwidth]{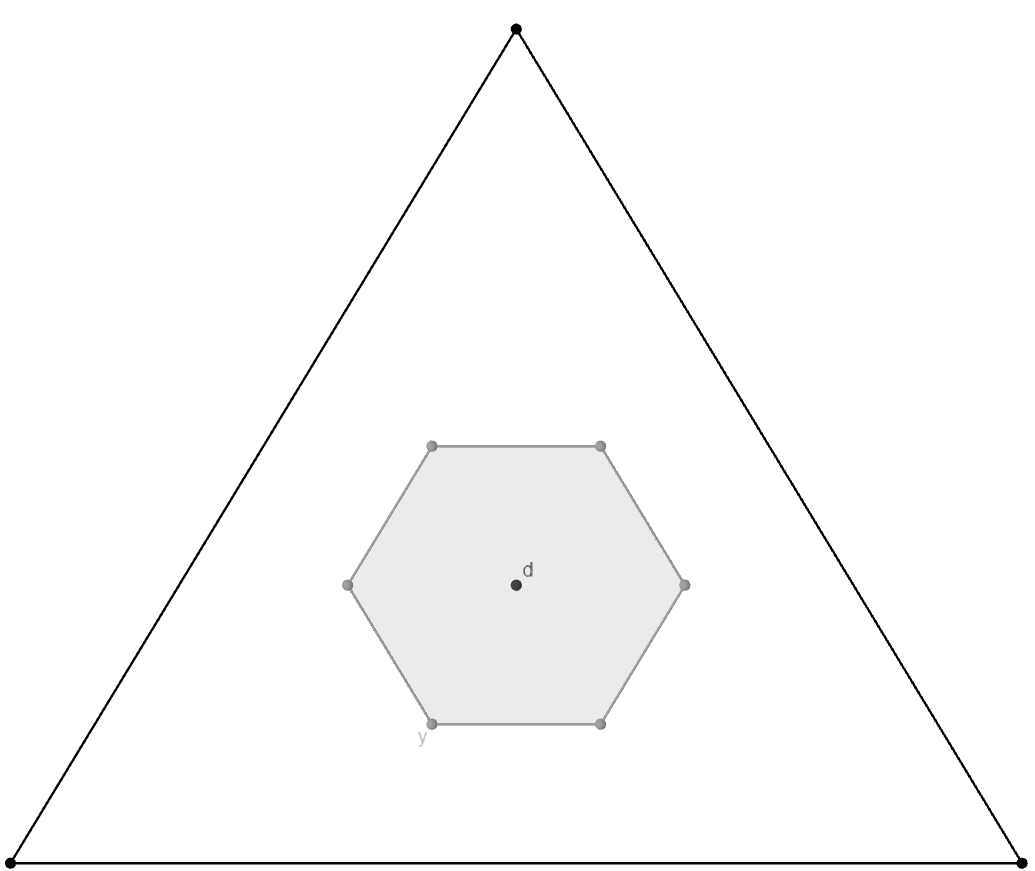}
(ii)\includegraphics[width=0.45\textwidth]{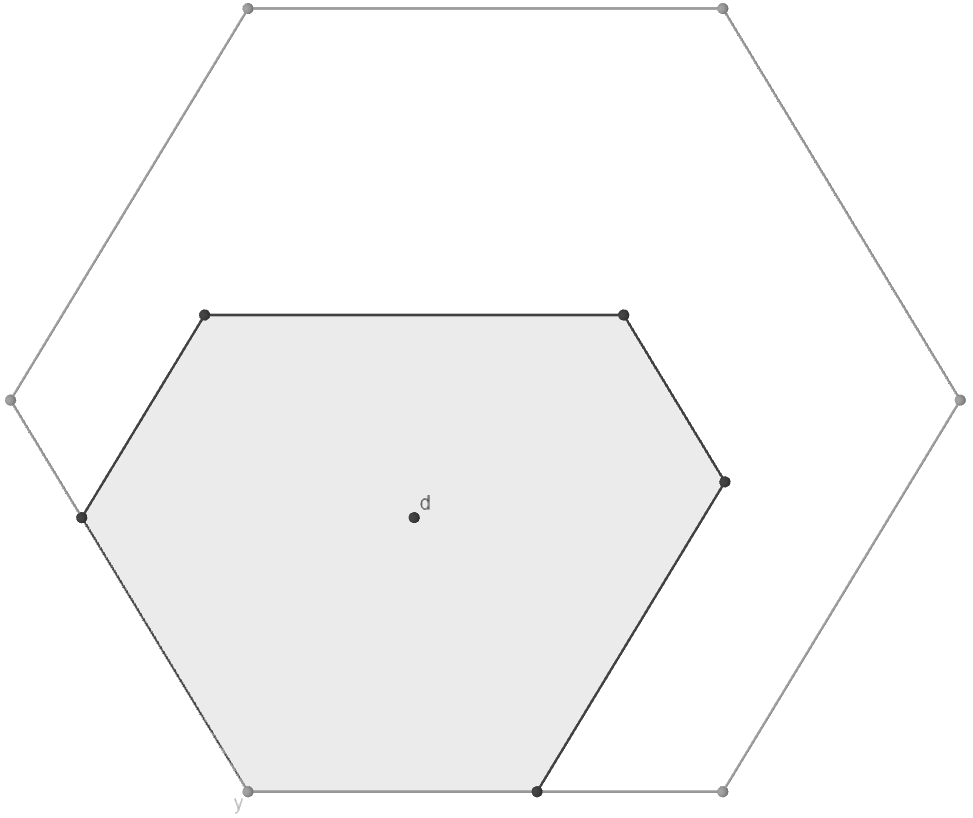}
(iii)\includegraphics[width=0.45\textwidth]{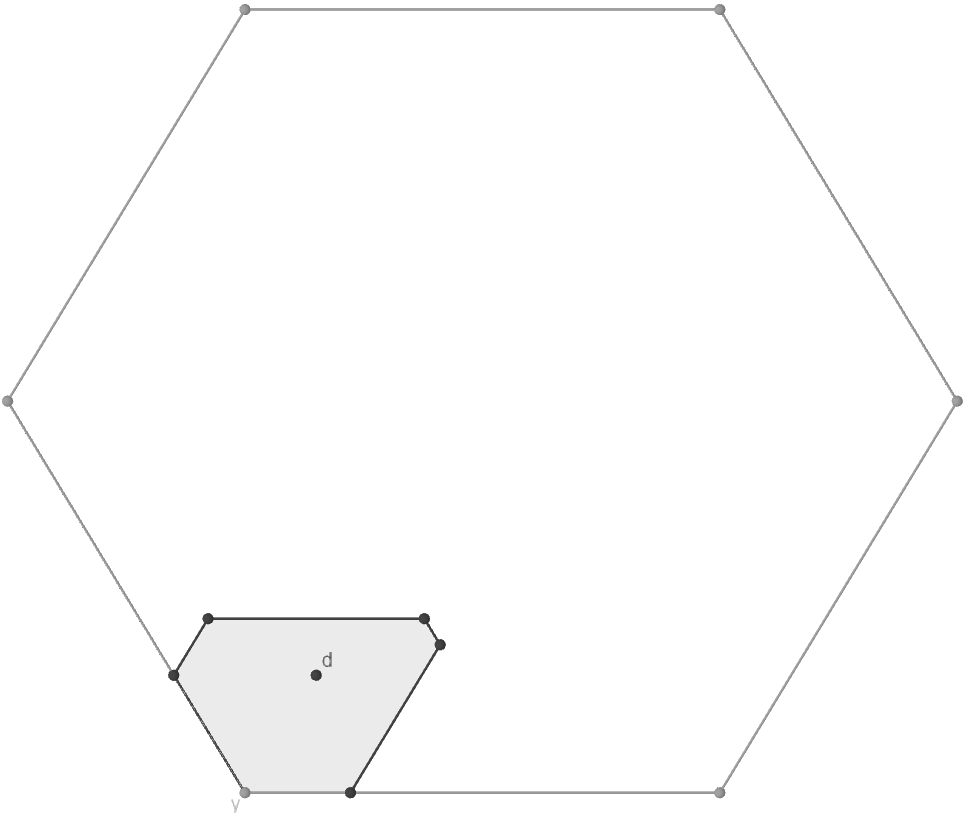}
(iv)\includegraphics[width=0.45\textwidth]{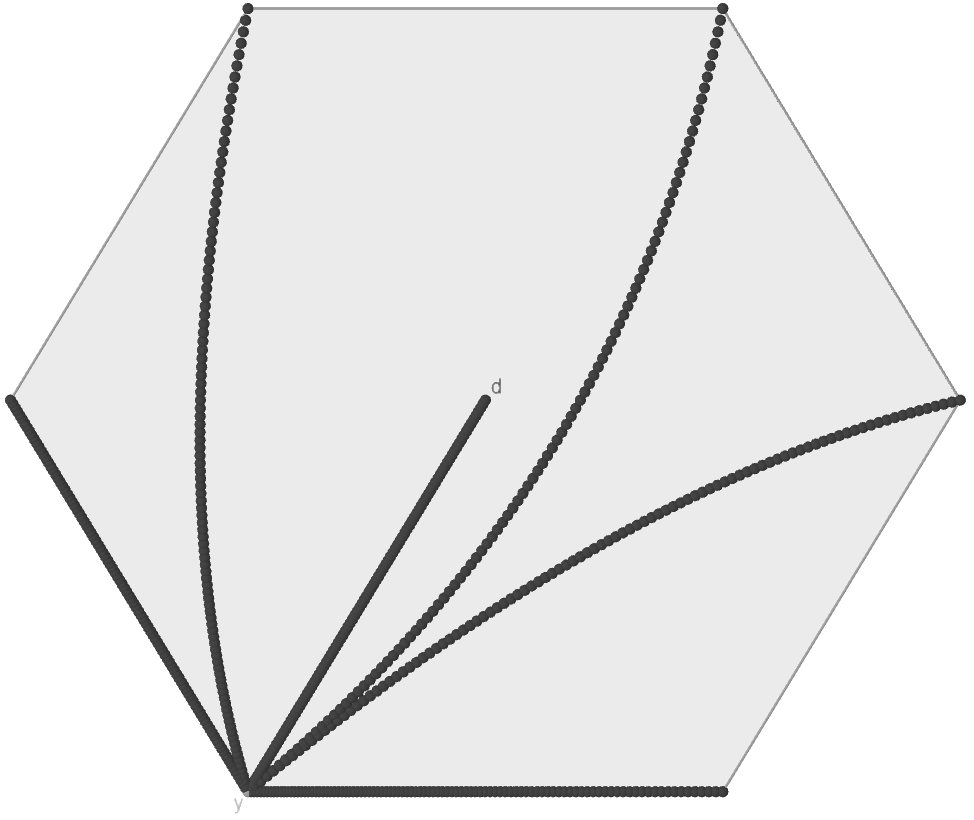}
\caption{%\textit{(Color online).} 
Visualization of Example \ref{ex_wandering_d_vector}. 
(i): Shows $M_{d(0)}(y)=M_\unitvector(y)=\operatorname{conv}\{\underline{\sigma}y : \sigma\in S_3\}$ inside (a multiple of) the 3-dimensional standard simplex. (ii): Zooms in on the classical majorization polytope $M_\unitvector(y)$. The shaded area is $M_{d(\lambda)}(y)$ for $\lambda=0.3$. (iii): Shows $M_{d(\lambda)}(y)$ for $\lambda=0.7$. (iv): The graph of the map $\lambda\mapsto\operatorname{ext}(M_{d(\lambda)}(y))$.}\label{Abb2}
\end{figure}
%%ARXIV ONLY
%
%

%% If you have bibdatabase file and want bibtex to generate the
%% bibitems, please use
%%
%% \bibliographystyle{elsarticle-num} 
%% \bibliography{<your bibdatabase>}

\noindent\textbf{Acknowledgments.} The authors are grateful to Narutaka Ozawa for providing the counterexample in Remark \ref{rem_maj_vector} (iv) in a discussion on MathOverflow \cite{322636}, as well as the anonymous referee for making them aware of reference \cite{Dahl10} and for beneficial remarks which led to a substantially improved presentation of the material.
The authors also thank Thomas Schulte-Herbr\"uggen
for valuable and constructive comments during the preparation of this manuscript, and Matteo Lostaglio and {\'{A}}lvaro M.~Alhambra for drawing their attention to some more recent publications related to thermo-majorization.
This research is part of the Bavarian excellence network \textsc{enb}
via the International PhD Programme of Excellence
\textit{Exploring Quantum Matter} (\textsc{exqm}), as well as the \textit{Munich Quantum Valley} of the Bavarian
State Government with funds from Hightech Agenda \textit{Bayern Plus}.

\bibliographystyle{elsarticle-harv}
\bibliography{../../../../../../../../control21vJan20}
%% else use the following coding to input the bibitems directly in the
%% TeX file.

%\begin{thebibliography}{00}
%
%%% \bibitem{label}
%%% Text of bibliographic item
%
%\bibitem{}
%
%\end{thebibliography}
\end{document}